\documentclass[11pt, a4paper]{amsart}

\usepackage{amsmath, amsthm, amssymb}

\newcommand{\bP}{\mathbb{P}}
\newcommand{\bC}{\mathbb{C}}
\newcommand{\diag}{\operatorname{diag}}
\newcommand{\rd}{\mathrm{d}}
\newcommand{\bN}{\mathbb{N}}
\newcommand{\bR}{\mathbb{R}}
\newcommand{\supp}{\operatorname{supp}}
\newcommand{\Res}{\operatorname{Res}}
\newcommand{\Id}{\mathrm{Id}}
\newcommand{\sP}{\mathsf{P}}
\newcommand{\sA}{\mathsf{A}}
\newcommand{\cS}{\mathcal{S}}
\newcommand{\cZ}{\mathcal{Z}}
\newcommand{\can}{\operatorname{can}}
\newcommand{\cE}{\mathcal{E}}
\newcommand{\sd}{\mathsf{d}}
\newcommand{\Lip}{\operatorname{Lip}}

\newcommand{\PGL}{\mathit{PGL}}

\newcommand{\diam}{\operatorname{diam}}
\newcommand{\ord}{\operatorname{ord}}

\newcommand{\cB}{\mathcal{B}}
\newcommand{\sH}{\mathsf{H}}

\newcommand{\sJ}{\mathsf{J}}
\newcommand{\sF}{\mathsf{F}}
\newcommand{\bZ}{\mathbb{Z}}
\newcommand{\cO}{\mathcal{O}}

\theoremstyle{plain}
\newtheorem{theorem}{Theorem}[section]
\newtheorem{lemma}[theorem]{Lemma}
\newtheorem{proposition}[theorem]{Proposition}
  
\newtheorem{mainth}{Theorem}

\theoremstyle{definition}
\newtheorem{definition}[theorem]{Definition}
\newtheorem{notation}[theorem]{Notation}
\newtheorem*{question}{Question}

\theoremstyle{remark}
\newtheorem{remark}[theorem]{Remark}
\newtheorem{fact}[theorem]{Fact}
\newtheorem{example}[theorem]{Example}

\numberwithin{equation}{section}

\begin{document} 

\title[Algebraic zeros divisors on the projective line]{
Algebraic zeros divisors on the projective line having small diagonals and 
small heights and their application to adelic dynamics}

\author[Y\^usuke Okuyama]{Y\^usuke Okuyama}
\address{
Division of Mathematics,
Kyoto Institute of Technology,
Sakyo-ku, Kyoto 606-8585 Japan.}
\email{okuyama@kit.ac.jp}

\date{\today}


\begin{abstract}
We establish a quantitative adelic equidistribution theorem
for a sequence of algebraic zeros divisors on the projective line
over the separable closure of a product formula field
having small diagonals and 
small $g$-heights with respect to an adelic normalized
weight $g$ in arbitrary characteristic and
in possibly non-separable setting, and obtain
local proximity estimates 
between the iterations of a rational function $f\in k(z)$ of degree $>1$
and a rational function $a\in k(z)$ of degree $>0$ 
over a product formula field $k$
of characteristic $0$, applying this quantitative adelic equidistribution
result to adelic dynamics of $f$.
\end{abstract}

\maketitle 


\section{Introduction}\label{sec:intro}

Let $k$ be a field and denote by $k_s$ the separable closure of $k$
in an algebraic closure $\overline{k}$ of $k$. For every $d\in\bN\cup\{0\}$,
let $k[p_0,p_1]_d$ be the set of all homogeneous polynomials of
two variables over $k$ of degree $d$.
A $k$-{\itshape algebraic
zeros divisor} $\cZ$ on $\bP^1(\overline{k})$ is a divisor on $\bP^1(\overline{k})$
defined by the zeros in $\bP^1(\overline{k})$
of a
$P\in\bigcup_{d\in\bN}k[p_0,p_1]_d$
{\itshape taking into account their multiplicities}, and
is said to be on $\bP^1(k_s)$ if $\supp\cZ\subset\bP^1(k_s)$:
the defining polynomial $P(p_0,p_1)$ of $\cZ$
is unique up to multiplication in $k^*(=k\setminus\{0\})$, and
is called a {\itshape representative}
of $\cZ$. Algebraic zeros divisors include Galois conjugacy classes of
algebraic numbers, and are also called {\itshape Galois stable multisets}
in $\bP^1(\overline{k})$.

Our aims in this article are to establish a
{\itshape quantitative} adelic equidistribution
of a sequences of $k$-algebraic zeros divisors
on $\bP^1(k_s)$, where $k$ is a {\itshape product formula} field,
having not only small $g$-heights 
(with respect to an adelic normalized weight $g$) 
but also {\itshape small diagonals} 
in arbitrary characteristic and in possibly non-separable setting, and
to contribute to the study of the local {\itshape proximities} 
between the iterations of a rational function $f\in k(z)$ of degree $>1$
and a rational function $a\in k(z)$ of degree $>0$
on a chordal disk $D$ of radius $>0$ in the projective line $\bP^1(\bC_v)$
for each place $v$ of $k$, in the setting of adelic dynamics of characteristic
$0$.

\subsection{Arithmetic over a product formula field}
A field $k$ is a {\itshape product formula field} if $k$ is equipped with
(i) a set $M_k$ of all places of $k$, which are either {\itshape finite}
or {\itshape infinite}, (ii) a set $\{|\cdot|_v:v\in M_k\}$,
where for each $v\in M_k$, $|\cdot|_v$ is a non-trivial absolute value
of $k$ representing $v$
(and then by definition the
$|\cdot|_v$ is non-archimedean if and only if $v$ is finite),
and (iii) a set $\{N_v:v\in M_k\}$,
where $N_v\in\bN$ for every $v\in M_k$,
such that the following {\itshape product formula} holds:
for every $z\in k\setminus\{0\}$, $|z|_v\neq 1$ for
at most finitely many $v\in M_k$ and
\begin{gather} 
 \prod_{v\in M_k}|z|_v^{N_v}=1.\tag{PF}\label{eq:product}
\end{gather}
Product formula fields include number fields
and function fields over curves, and a product formula field $k$ is a
number field if and only if there is at least one infinite place of $k$
(see, e.g., the paragraph after \cite[Definition 7.51]{BR10}). 

Let $k$ be a product formula field. For each $v\in M_k$, let $k_v$ be the completion
of $k$ with respect to $|\cdot|_v$ and $\bC_v$ the completion of an algebraic
closure $\overline{k}_v$ of $k_v$ with respect to (the extended) $|\cdot|_v$, and
we fix an embedding of $\overline{k}$ to $\bC_v$ which extends that of $k$ to $k_v$:
by convention, the dependence of a local quantity
induced by $|\cdot|_v$ on each $v\in M_k$ is emphasized 
by adding the suffix $v$ to it.
A family $g=\{g_v:v\in M_k\}$ is an {\itshape adelic continuous weight}
if (i) for every $v\in M_k$, $g_v$ is a continuous function
on the {\itshape Berkovich} projective line $\sP^1(\bC_v)$ such that
\begin{gather*}
 \mu_v^g:=\Delta g_v+\Omega_{\can,v}
\end{gather*}
is a probability Radon measure on $\sP^1(\bC_v)$
(see \eqref{eq:spherical} and \eqref{eq:Laplacian} for the definition of
the probability Radon measure $\Omega_{\can,v}$ on $\sP^1(\bC_v)$
and the normalization of the Laplacian $\Delta$
on $\sP^1(\bC_v)$, respectively) and (ii)
there is a finite subset $E_g$ in $M_k$ such that 
$g_v\equiv 0$ on $\sP^1(\bC_v)$
for every $v\in M_k\setminus E_g$, and is still called 
an {\itshape adelic normalized weight} if (iii) in addition
the $g_v$-{\itshape equilibrium energy} $V_{g_v}$ of $\sP^1(\bC_v)$
vanishes for every $v\in M_k$
(see \S\ref{th:potential} for the definition of $V_{g_v}$).
For an adelic continuous weight $g=\{g_v:v\in M_k\}$,
the family $\mu^g:=\{\mu^g_v:v\in M_k\}$ is called
an {\itshape adelic probability measure}
(cf.\ \cite[D\'efinition 1.1]{FR06}).
An adelic continuous weight
$g=\{g_v:v\in M_k\}$ is said to be {\itshape placewise H\"older continuous} if
for every $v\in M_k$, $g_v$ is H\"older continuous on $\sP^1(\bC_v)$
with respect to the small model metric $\rd_v$ on $\sP^1(\bC_v)$
(see \eqref{eq:small} for the definition of $\sd_v$).

The $g$-{\itshape height} of
a $k$-algebraic zeros divisor $\cZ$ on $\bP^1(\overline{k})$
represented by a $P\in\bigcup_{d\in\bN}k[p_0,p_1]_d$ with respect to
an adelic continuous weight $g=\{g_v:v\in M_k\}$ is
\begin{gather}
 h_g(\cZ):=\sum_{v\in M_k}N_v\frac{M_{g_v}(P)}{\deg P},\label{eq:definingheight}
\end{gather}
where for every $v\in M_k$, $M_{g_v}(P)$ is the logarithmic
$g_v$-{\itshape Mahler measure} of the $P$
(see \eqref{eq:Mahler} for the definition of $M_{g_v}(P)$ and
\S\ref{sec:finiteness} for a proof of $h_g(\cZ)\in\bR$);
by \eqref{eq:product}, $h_g(\cZ)$ is well defined.
For every $v\in M_k$,
letting $\delta_\cS$ be the Dirac measure on
$\sP^1(\bC_v)$ at a point $\cS\in\sP^1(\bC_v)$,
a $k$-algebraic zeros divisor $\cZ$ on $\bP^1(\overline{k})$ is
regarded as a positive and discrete Radon measure
$\sum_{w\in\supp\cZ}(\ord_w\cZ)\delta_w$
on $\sP^1(\bC_v)$, which is still denoted by $\cZ$. Then
the {\itshape diagonal}
\begin{gather*}
 (\cZ\times\cZ)(\diag_{\bP^1(\overline{k})})=\sum_{w\in\supp\cZ}(\ord_w\cZ)^2
\end{gather*}
of $\cZ$ is independent of $v\in M_k$.
For a sequence $(\cZ_n)$ of $k$-algebraic zeros divisors on
$\bP^1(\overline{k})$ satisfying $\lim_{n\to\infty}\deg\cZ_n=\infty$,
we say $(\cZ_n)$
{\itshape has small $g$-heights
with respect to an adelic normalized weight $g$} 
if $\limsup_{n\to\infty}h_g(\cZ_n)\le 0$, and
say $(\cZ_n)$ {\itshape has small diagonals} if
$\lim_{n\to\infty}((\cZ_n\times\cZ_n)(\diag_{\bP^1(\overline{k})}))/(\deg\cZ_n)^2=0$.

\subsection{Quantitative adelic equidistribution of algebraic zeros divisors}
The following is one of our principal results:
for the Galois conjugacy class of an algebraic number,
this was due to Favre--Rivera-Letelier \cite[Th\'eor\`eme 7]{FR06}.
For the definitions of the $C^1$-regularity of a continuous test function
$\phi$ on $\sP^1(\bC_v)$,
the Lipschitz constant $\Lip(\phi)_v$ on $(\sP^1(\bC_v),\rd_v)$, and
the Dirichlet norm $\langle\phi,\phi\rangle_v$ of $\phi$
for each $v\in M_k$, see Section \ref{sec:quantitative}.

\begin{mainth}\label{th:adelicquantitative}
Let $k$ be a product formula field 
and $k_s$ the separable closure of $k$ in $\overline{k}$.
Let $g=\{g_v:v\in M_k\}$ be a placewise H\"older continuous adelic normalized weight.
Then for every $v\in M_k$, there is $C>0$ such that 
for every $k$-algebraic zeros divisor $\cZ$ on $\bP^1(k_s)$
and every test function $\phi\in C^1(\sP^1(\bC_v))$,
\begin{multline}
\left|
\int_{\sP^1(\bC_v)}\phi\rd\left(\frac{\cZ}{\deg\cZ}-\mu^g_v\right)
\right|
\le \\
C\cdot\max\{\Lip(\phi)_v,\langle\phi,\phi\rangle_v^{1/2}\}
\sqrt{\max\left\{h_g(\cZ),
(\log\deg\cZ)\frac{(\cZ\times\cZ)(\diag_{\bP^1(k_s)})}{(\deg\cZ)^2}\right\}}.\label{eq:quantitative}
\end{multline}
\end{mainth}

In Theorem \ref{th:adelicquantitative},
if $v\in M_k$ is an infinite place, 
or equivalently, $\bC_v\cong\bC$,
then the estimate \eqref{eq:quantitative} gives a quantitative estimate of the
{\itshape Kantorovich--Wasserstein metric} 
\begin{gather*}
 W\left(\frac{\cZ}{\deg\cZ},\mu^g_v\right)
 =\sup_{\phi}\left|
 \int_{\bP^1(\bC)}\phi\rd\left(\frac{\cZ}{\deg\cZ}-\mu^g_v\right)\right|
\end{gather*}
between the probability Radon measures $\cZ/\deg\cZ$ and $\mu^g_v$ on 
$\sP^1(\bC_v)\cong\bP^1(\bC)$,
where $\phi$ ranges over all Lipschitz continuous functions on $\bP^1(\bC)$
whose Lipschitz constants equal $1$ with respect to 
the normalized chordal metric $[z,w]$ on $\bP^1(\bC)$
(see Remark \ref{th:dense}). 
For the details of the metric $W$ including its role in
the optimal transportation problems, see, e.g., \cite{Villani09}.

The following is a qualitative version of Theorem \ref{th:adelicquantitative}:
for a sequence of Galois conjugacy classes of algebraic numbers, this was due to
Baker--Rumely \cite[Theorem 2.3]{BR06}, 
Chambert-Loir \cite[Th\'eor\`eme 4.2]{ChambertLoir06},
Favre--Rivera-Letelier \cite[Th\'eor\`eme 2]{FR06};
see also Szpiro--Ullmo--Zhang \cite{SUZ97}, 
Bilu \cite{Bilu97}, Rumely \cite{Rumely99}, Chambert-Loir
\cite{ChambertLoir00}, Autissier \cite{Autissier01}, 
Baker--Hsia \cite{BakerHsia05}, Baker--Rumely \cite{BR06}, 
Chambert-Loir \cite{ChambertLoir06},
Favre--Rivera-Letelier \cite{FR06}, and, ultimately, Yuan \cite{Yuan08}.

\begin{mainth}[asymptotically Fekete configuration of algebraic zeros divisors]\label{th:arith}
Let $k$ be a product formula field 
and $k_s$ the separable closure of $k$ in $\overline{k}$.
Let $g=\{g_v:v\in M_k\}$ be an adelic normalized weight. If
a sequence $(\cZ_n)$ of $k$-algebraic zeros divisors on $\bP^1(k_s)$ 
satisfying $\lim_{n\to\infty}\deg\cZ_n=\infty$
has both small diagonals
and small $g$-heights, then
for every $v\in M_k$, $(\cZ_n)$ is an 
asymptotically $g_v$-Fekete configuration on $\sP^1(\bC_v)$.
In particular,
$\lim_{n\to\infty}\cZ_n/\deg\cZ_n=\mu^g_v$ weakly on $\sP^1(\bC_v)$.
\end{mainth}

In Theorem \ref{th:arith},
the assertion that {\itshape $(\cZ_n)$ is
an asymptotically $g_v$-Fekete configuration on $\sP^1(\bC_v)$}
(see \eqref{eq:Feketeequiv} for the definition), which is also called a
$g_v$-{\itshape pseudo-equidistribution} on $\sP^1(\bC_v)$, 
is stronger than the final equidistribution assertion.
For a relationship between the Kantorovich--Wasserstein metric $W$ and 
$($asymptotically$)$ Fekete configurations
on complex manifolds, 
see Lev and Ortega-Cerd\`a \cite[\S 7]{LevOrtega12}. 
For a recent result on the {\itshape capacity and the transfinite diameter} 
on complex manifolds, see Berman--Boucksom \cite{BermanBoucksom10} 
(on $\bC^n$,
we also refer to the survey \cite{Levenberg10}) and
the {\itshape convergence of $($asymptotically$)$ Fekete points}
on complex manifolds,
see Berman--Boucksom--Nystr\"om \cite{BBN11}.

\subsection{Quantitative equidistribution in adelic dynamics}
For rational functions $f,a\in k(z)$ over a field $k$ and every $n\in\bN$, 
the divisor $[f^n=a]$
defined by the roots of the equation $f^n=a$ in $\bP^1(\overline{k})$ is
a $k$-algebraic zeros divisor on $\bP^1(\overline{k})$
if $f^n\not\equiv a$.

Let $k$ be a product formula field. For a rational function $f\in k(z)$
of degree $d>1$, let $\hat{g}_f:=\{g_{f,v}:v\in M_k\}$ be the {\itshape adelic
dynamical Green function} in that for every $v\in M_k$, $g_{f,v}$ is
the dynamical Green function of $f$ on $\sP^1(\bC_v)$, so that
$\mu_{f,v}:=\mu^{g_{f,v}}$ is the $f$-equilibrium (or canonical) measure on
$\sP^1(\bC_v)$ 
(see Section \ref{subsec:adelicGreen} for the details): the family $\hat{g}_f$ is indeed an adelic normalized weight, and the $\hat{g}_f$-height function $h_{\hat{g}_f}$
coincides with the Call-Silverman
$f$-{\itshape dynamical} (or canonical) height function. For every rational function $a\in k(z)$,
the sequence $([f^n=a])$ has {\itshape strictly small} 
$\hat{g}_f$-heights in that
$\limsup_{n\to\infty}(d^n+\deg a)\cdot h_{\hat{g}_f}([f^n=a])<\infty$
(Lemma \ref{th:heightsdynamics}). Hence the following 
are consequences of
Theorems \ref{th:adelicquantitative} and \ref{th:arith},
respectively.

\begin{mainth}\label{th:quantitativedynamics}
Let $k$ be a product formula field 
and $k_s$ the separable closure of $k$ in $\overline{k}$. 
Let $f\in k(z)$ be a rational function of degree $d>1$ and $a\in k(z)$
a rational function.
Then for every $v\in M_k$,
there exists a constant $C>0$ such that
for every test function $\phi\in C^1(\sP^1(\bC_v))$
and every $n\in\bN$,
\begin{multline}
\left|
\int_{\sP^1(\bC_v)}\phi\rd\left(\frac{[f^n=a]}{d^n+\deg a}-\mu_{f,v}\right)
\right|\\
\le 
C\cdot\max\{\Lip(\phi)_v,\langle\phi,\phi\rangle^{1/2}_v\}
\sqrt{\frac{n\cdot([f^n=a]\times[f^n=a])(\diag_{\bP^1(k_s)})}{(d^n+\deg a)^2}}\label{eq:error}
\end{multline}
if $f^n\not\equiv a$ and the divisor $[f^n=a]$ on $\bP^1(\overline{k})$ is on $\bP^1(k_s)$. 
\end{mainth}

\begin{mainth}\label{th:qualitativedynamics}
Let $k$ be a product formula field 
and $k_s$ the separable closure of $k$ in $\overline{k}$. 
Let $f\in k(z)$ be a rational function of degree $d>1$ and
$a\in k(z)$ a rational function. 
If the sequence $([f^n=a])$ has small diagonals and 
the divisor $[f^n=a]$ is on $\bP^1(k_s)$ for every $n\in\bN$ large enough, then for every $v\in M_k$, 
$([f^n=a])$ is an asymptotically $g_{f,v}$-Fekete configuration on $\sP^1(\bC_v)$.
In particular, $\lim_{n\to\infty}[f^n=a]/(d^n+\deg a)=\mu_{f,v}$ weakly on 
$\sP^1(\bC_v)$.
\end{mainth} 

The final equidistribution assertion in Theorem \ref{th:qualitativedynamics}
has been established by Brolin \cite{Brolin},
Lyubich \cite{Lyubich83}, Freire-Lopes-Ma\~n\'e \cite{FLM83}
in complex dynamics, and
Favre--Rivera-Letelier \cite{FR09} in (not necessarily adelic)
non-archimedean dynamics (of characteristic $0$ when $\deg a>0$).
For every constant $a\in\bP^1(k)$, the estimate \eqref{eq:error} in
Theorem \ref{th:quantitativedynamics} has been obtained in
\cite[Theorems 4 and 5]{OkuFekete} in complex and
(not necessarily adelic) non-archimedean dynamics of characteristic $0$.
In complex dynamics, for every $f\in\bC(z)$ of degree $d>1$, every 
constant $a\in\bP^1(\bC)$,
and every $\phi\in C^2(\bP^1(\bC))$, a finer estimate
than \eqref{eq:error}
has been obtained by \cite[Theorem 2 together with (4.2)]{DOproximity}.

\subsection{Application to a motivating Question}

Let $K$ be an algebraically closed field that is complete with respect to a
non-trivial absolute value $|\cdot|$, and $[z,w]$ be
the normalized {\itshape chordal metric} on $\bP^1=\bP^1(K)$ (see \eqref{eq:chordaldist}). A subset $D$ in $\bP^1$ is called 
a {\itshape chordal disk}
(in $\bP^1$)
if $D=\{z\in\bP^1:[z,w]\le r\}$ for some $w\in\bP^1$ and some {\itshape radius}
$r\ge 0$.
Even in the specific case $a=\Id$
(see, e.g., Cremer \cite{Cremer28}, Siegel \cite{Siegel}, Brjuno \cite{Brjuno},
Herman--Yoccoz \cite{HY83},
Yoccoz \cite{Yoccoz88, Yoccoz96}, P\'erez-Marco \cite{PMar, PerezMarco01}),
which is one of the most interesting cases and
is related to {\itshape the difficulty of small denominators}
in non-archimedean and complex dynamics,
the following has not been completely understood:
\begin{question}
How uniformly close on a chordal disk $D$ of radius $>0$
can the sequence $(f^n)$ of the iterations 
of a rational function $f\in K(z)$ of degree $>1$
be to a rational function $a\in K(z)$ of degree $>0$?
\end{question}
For a study of this question
on the projective space $\bP^N(K)$, see \cite{OkuNonlinear}.
The following estimate of
the {\itshape local proximity sequence}
$(\sup_D[f^n,a]_v)$ is an application of Theorem \ref{th:quantitativedynamics}
to this Question in the setting of adelic dynamics.

\begin{mainth}\label{th:nonlinearity}
 Let $k$ be a product formula field of characteristic $0$.
 Let $f\in k(z)$ be a rational function 
 of degree $>1$ and $a\in k(z)$ a rational function of degree $>0$. 
 Then for every $v\in M_k$ and every chordal disk $D$ in $\bP^1(\bC_v)$
 of radius $>0$, as $n\to\infty$,
\begin{gather}  
 \log\sup_D[f^n,a]_v=O\left(\sqrt{n\cdot([f^n=a]\times[f^n=a])(\diag_{\bP^1(\overline{k})})}\right).
\label{eq:nonlin}
\end{gather}
Here, the implicit constant in $O(\cdot)$ possibly depends on $f$ and $a$.
\end{mainth}

In the case that $a=\Id$, we will see that
$([f^n=\Id]\times[f^n=\Id])(\diag_{\bP^1(\overline{k})})=O(d^n)$ as $n\to\infty$
in Section \ref{eq:rotation}.
Hence Theorem \ref{th:nonlinearity} concludes the following.

\begin{mainth}\label{th:identity}
Let $k$ be a product formula field of characteristic $0$.
Let $f\in k(z)$ be a rational function of degree $d>1$. 
Then for every $v\in M_k$ and every chordal disk $D$ in $\bP^1(\bC_v)$
of radius $>0$,
\begin{gather}
 \log\sup_D[f^n,\Id]_v=O(\sqrt{nd^n})\quad\text{as }n\to\infty.\label{eq:recurrence} 
\end{gather}
\end{mainth}

\subsection{The unit $D^*(p)$}
The following generalizes the obvious fact that the discriminant of a
polynomial in one variable over a field $k$ is in $k$: 
the unit $D^*(p)$ plays an important role in non-separable case
and might have ever been studied, but we could find no literature.

\begin{mainth}\label{th:discriminant}
 Let $k$ be a field and $k_s$ the separable closure of $k$ in an algebraic closure $\overline{k}$ of $k$.
 For every $p(z)\in k[z]$ of degree $>0$, let $\{z_1,\ldots,z_m\}$
 be the set of all distinct zeros of $p(z)$ in $\overline{k}$ so that
 $p(z)=a\cdot\prod_{j=1}^m(z-z_j)^{d_j}$ in $\overline{k}[z]$
 for some $a\in k\setminus\{0\}$ and some sequence $(d_j)_{j=1}^m$ in $\bN$.
 If $\{z_1,\ldots,z_m\}\subset k_s$, then
\begin{gather*}
 D^*(p)
 :=\textstyle\prod_{j=1}^m\prod_{i:\, i\neq j}(z_j-z_i)^{d_id_j}\in k\setminus\{0\}, 
\end{gather*} 
where, {\rmfamily a priori}, this $D^*(p)$ is always
 in $\overline{k}\setminus\{0\}$.
\end{mainth}

\subsection{Organization of this article}
In Section  \ref{sec:background}, we recall a background
from potential theory
and arithmetic on the Berkovich projective line.
In Section \ref{sec:regularization}, 
we extend Favre--Rivera-Letelier's regularization 
$[\cdot]_\epsilon$ of discrete Radon measures
and establish required estimates
on them, and in Section \ref{sec:CS} we see the negativity of
regularized Fekete sums and a Cauchy-Schwarz inequality.
In Sections \ref{sec:algebraic} and \ref{sec:regularized},
we compute the $g$-Fekete sums $(\cZ,\cZ)_g$
and estimate the regularized $g$-Fekete sums $(\cZ_\epsilon,\cZ_\epsilon)_g$
with respect to a $k$-algebraic zeros divisor $\cZ$ on $\bP^1(\overline{k})$.
In Section \ref{sec:quantitative}, we show
Theorems \ref{th:adelicquantitative} and \ref{th:arith};
the arguments
are more or less adaptions of those in the proofs of
Favre--Rivera-Letelier \cite[Th\'eor\`eme 7]{FR06}
and Baker--Rumely \cite[Theorem 10.24]{BR10}, respectively.
In Section \ref{sec:dynamics}, we recall a background from
non-archimedean and complex dynamics. In Section \ref{subsec:adelicGreen},
we show Theorems \ref{th:quantitativedynamics} and \ref{th:qualitativedynamics},
and in Section \ref{eq:rotation} we show Theorems
\ref{th:nonlinearity} and \ref{th:identity}.
In Section \ref{sec:discriminant},
Theorem \ref{th:discriminant} is shown.

\section{Background from potential theory and arithmetic}\label{sec:background}

\begin{notation}
For a field $k$, the origin of $k^2$ is also denoted by $0=0_k$, and 
$\pi=\pi_k:k^2\setminus\{0\}\to\bP^1=\bP^1(k)$ is the canonical projection
so that $\pi(p_0,p_1)=p_1/p_0$ if $p_0\neq 0$ and that $\pi(0,1)=\infty$.
Set the wedge product $(z_0,z_1)\wedge(w_0,w_1):=z_0w_1-z_1w_0$ 
on $k^2\times k^2$.
\end{notation}

Let $K$ be an algebraically closed field that is complete with respect to a
non-trivial absolute value $|\cdot|$,
which is said to be {\itshape non-archimedean}
if the strong triangle inequality $|z+w|\le\max\{|z|,|w|\}$ on $K\times K$
holds, and to be {\itshape archimedean} unless $K$ is non-archimedean.
On $K^2$, let $\|(p_0,p_1)\|$ be either 
the maximal norm
$\max\{|p_0|,|p_1|\}$ 
(for non-archimedean $K$) 
or the Euclidean norm $\sqrt{|p_0|^2+|p_1|^2}$ (for archimedean $K$).
The {\itshape normalized chordal metric} $[z,w]$ on $\bP^1=\bP^1(K)$ is a function
\begin{gather}
 (z,w)\mapsto [z,w]=|p\wedge q|/(\|p\|\cdot\|q\|)\le 1\label{eq:chordaldist}
\end{gather}
on $\bP^1\times\bP^1$, where $p\in\pi^{-1}(z),q\in\pi^{-1}(w)$.
The metric topology on $\bP^1$ with respect to $[z,w]$ agrees with
the relative topology on $\bP^1$ from the {\itshape Berkovich} projective line
$\sP^1=\sP^1(K)$, which is a compact augmentation of $\bP^1$
and contains $\bP^1$ as a dense subset, and is isomorphic to $\bP^1$ if and only if
$K$ is archimedean (see \S\ref{sec:tree} for more details when $K$ is non-archimedean).
Letting $\delta_{\cS}$ be the Dirac measure on $\sP^1$ at a point $\cS\in\sP^1$,
set 
\begin{gather}
 \Omega_{\can}:=\begin{cases}
		\delta_{\cS_{\can}} & \text{for non-archimedean } K,\\
		\omega & \text{for archimedean } K,
	       \end{cases}\label{eq:spherical}
\end{gather}
where $\cS_{\can}$ is the canonical (or Gauss) point in $\sP^1$
for non-archimedean $K$ (see \S\ref{sec:tree} for the definition of it),
and $\omega$ is the Fubini-Study area element on $\bP^1$
normalized as $\omega(\bP^1)=1$ for archimedean $K$.
For non-archimedean $K$,
the {\itshape generalized Hsia kernel} $[\cS,\cS']_{\can}$ on $\sP^1$
with respect to $\cS_{\can}$ 
is the unique (jointly) upper semicontinuous and 
separately continuous extension of the normalized chordal metric
$[z,w]$ on $\bP^1(\times\bP^1)$ to $\sP^1\times\sP^1$ (see \eqref{eq:Gromov}
for more concrete description).
By convention,
for archimedean $K$, the kernel function $[\cS,\cS']_{\can}$ is defined by the $[z,w]$
itself.
Let $\Delta=\Delta_{\sP^1}$ be the distributional Laplacian on $\sP^1$ normalized
so that for each $\cS'\in\sP^1$, 
\begin{gather}
 \Delta\log[\cdot,\cS']_{\can}=\delta_{\cS'}-\Omega_{\can}\quad\text{on }\sP^1.\label{eq:Laplacian}
\end{gather}
For the construction of the Laplacian $\Delta$ in non-archimedean case,
see \cite[\S 5]{BR10}, \cite[\S7.7]{FJbook}, \cite[\S 3]{ThuillierThesis}
and also \cite[\S2.5]{Jonsson15}.

\subsection{Potential theory on $\sP^1$ with external fields}\label{th:potential}
For the foundation of the potential theory on the (Berkovich) projective line,
see Baker--Rumely \cite{BR10}, Favre--Rivera-Letelier \cite{FR09},
Thuillier \cite{ThuillierThesis}, and
also Jonsson \cite{Jonsson15} and Tsuji \cite[III \S11]{Tsuji59}
(\cite{ThuillierThesis} is on more general curves than lines
and \cite[III \S11]{Tsuji59} is on $\bP^1(\bC)$).
We also refer to Saff--Totik \cite{ST97} 
for the generalities of {\itshape weighted} potential
theory, i.e., logarithmic potential theory {\itshape with external fields}.

A {\itshape continuous weight $g$ on $\sP^1$} is
a continuous function on $\sP^1$ such that
\begin{gather*}
 \mu^g:=\Delta g+\Omega_{\can} 
\end{gather*}
is a probability Radon measure on $\sP^1$.
For a continuous weight $g$ on $\sP^1$, the $g$-{\itshape potential kernel} on $\sP^1$
(or the negative of an Arakelov Green kernel function on $\sP^1$ relative to 
$\mu^g$ \cite[\S 8.10]{BR10}) is the function
 \begin{gather}
 \Phi_g(\cS,\cS'):=\log[\cS,\cS']_{\can}-g(\cS)-g(\cS') 
\quad\text{on }\sP^1\times\sP^1,\label{eq:kernel} 
 \end{gather}
 and the {\itshape $g$-potential}
 of a Radon measure $\nu$ on $\sP^1$ is the function
 \begin{gather}
 U_{g,\nu}(\cdot):=\int_{\sP^1}\Phi_g(\cdot,\cS')\rd\nu(\cS')\quad\text{on }\sP^1;\label{eq:potential}
 \end{gather}
 by the Fubini theorem, $\Delta U_{g,\nu}=\nu-\nu(\sP^1)\mu^g$ on $\sP^1$.
The {\itshape $g$-equilibrium energy $V_g\in[-\infty,+\infty)$ of}
$\sP^1$ 
is the supremum of the {\itshape $g$-energy} functional
\begin{gather}
 \nu\mapsto\int_{\sP^1\times\sP^1}\Phi_g\rd(\nu\times\nu)=\int_{\sP^1}U_{g,\nu}\rd\nu\label{eq:functional}
\end{gather}
on the space of all probability Radon measures $\nu$ on $\sP^1$; 
indeed, $V_g>-\infty$ since
$V_g\ge \int_{\sP^1\times\sP^1}\Phi_g\rd(\Omega_{\can}\times\Omega_{\can})
>-\infty$. 
A probability Radon measure $\mu$ on $\sP^1$ at which
the $g$-energy functional \eqref{eq:functional} attains the supremum $V_g$
is called a {\itshape $g$-equilibrium mass distribution on} $\sP^1$;
indeed, $\mu^g$ is the unique $g$-equilibrium mass distribution on $\sP^1$, and
moreover, $U_{g,\mu^g}\equiv V_g$ on $\sP^1$
(for non-archimedean $K$, see \cite[Theorem 8.67, Proposition 8.70]{BR10}).
For a discussion on such a Gauss variational problem, 
see Saff--Totik \cite[Chapter 1]{ST97}.
A {\itshape normalized weight $g$ on} $\sP^1$ 
is a continuous weight on $\sP^1$ satisfying $V_g=0$;
for every continuous weight $g$ on $\sP^1$, $\overline{g}:=g+V_g/2$ is
the unique normalized weight on $\sP^1$ such that $\mu^{\overline{g}}=\mu^g$.

For a continuous weight $g$ on $\sP^1$ and
a Radon measure $\nu$ on $\sP^1$,
the {\itshape $g$-Fekete sum} with respect to $\nu$ is
\begin{gather*}
 (\nu,\nu)_g:=\int_{\sP^1\times\sP^1\setminus\diag_{\bP^1(K)}}\Phi_g\rd(\nu\times\nu),
\end{gather*} 
which generalizes
the classical {\itshape Fekete sum} associated with a finite subset in $\bC$
(see \cite{Fekete30, Fekete33});
if $\supp\nu$ is a discrete (so finite) subset in $\sP^1$, i.e.,
$\nu$ is a {\itshape discrete} measure on $\sP^1$, then
$(\nu,\nu)_g$ is always finite (even if $\supp\nu\subset\bP^1$). 

For a continuous weight $g$ on $\sP^1$, 
a sequence $(\nu_n)$ of positive and discrete Radon measures on $\sP^1$
satisfying $\lim_{n\to\infty}\nu_n(\sP^1)=\infty$ is called an
{\itshape asymptotically $g$-Fekete configuration on} $\sP^1$
if $(\nu_n)$ not only has {\itshape small diagonals} in that
$(\nu_n\times\nu_n)(\diag_{\bP^1(K)})=o(\nu_n(\sP^1)^2)$
as $n\to\infty$ but also satisfies
$\lim_{n\to\infty}(\nu_n,\nu_n)_g/(\nu_n(\sP^1))^2=V_g$;
under the former small diagonals condition,
the latter one is indeed equivalent to the weaker
\begin{gather}
 \liminf_{n\to\infty}\frac{(\nu_n,\nu_n)_g}{(\nu_n(\sP^1))^2}\ge V_g\label{eq:Feketeequiv}
\end{gather} 
since we always have
\begin{gather}
 \limsup_{n\to\infty}\frac{(\nu_n,\nu_n)_g}{(\nu_n(\sP^1))^2}\le V_g\label{eq:Feketeupper}
\end{gather}
(see, e.g., \cite[Lemma 7.54]{BR10}).
By a classical argument (cf.\ \cite[Theorem 1.3 in Chapter III]{ST97}), 
if $(\nu_n)$ is an asymptotically $g$-Fekete configuration on $\sP^1$, then
$\lim_{n\to\infty}\nu_n/\nu_n(\sP^1)=\mu^g$ weakly on $\sP^1$. 

\subsection{Local arithmetic on $\sP^1$}\label{sec:local}

Let $k$ be a field. 

\begin{definition}
 A field extension $K/k$ is
 an {\itshape algebraic and metric augmentation} of $k$
 if $K$ is algebraically closed and (topologically) complete with respect to
 a non-trivial absolute value $|\cdot|$
 (e.g., the $\bC_v$ is an algebraic and metric augmentation of a product formula
 field $k$ for every $v\in M_k$).
\end{definition}

For every $P\in\bigcup_{d\in\bN}k[p_0,p_1]_d$,
there is a sequence $(q_j^P)_{j=1}^{\deg P}$  
in $\overline{k}^2\setminus\{0\}$ giving a factorization
\begin{gather}
 P(p_0,p_1)=\prod_{j=1}^{\deg P}((p_0,p_1)\wedge q_j^P)\label{eq:factorization}
\end{gather}
of $P$ in $\overline{k}[p_0,p_1]$. Set 
$z_j^P:=\pi(q_j^P)\in\bP^1(\overline{k})$ for each $j\in\{1,2,\ldots,\deg P\}$. 
Although the sequence $(q_j^P)_{j=1}^{\deg P}$ is not unique, 
the sequence $(z_j^P)_{j=1}^{\deg P}$ in $\bP^1(\overline{k})$
is independent of choices of $(q_j^P)_{j=1}^{\deg P}$ up to permutations.
Let in addition $K$ be an algebraic and metric completion of $k$.
Then the sum
$M^\#(P):=\sum_{j=1}^{\deg P}\log\|q_j^P\|$
is also independent of choices of $(q_j^P)_{j=1}^{\deg P}$, and
for every continuous weight $g$ on $\sP^1=\sP^1(K)$,
the {\itshape logarithmic $g$-Mahler measure} of $P$ is
\begin{gather}
 M_g(P):=\sum_{j=1}^{\deg P}g(z_j^P)+M^\#(P).\label{eq:Mahler}
\end{gather}

The function $S_P:=|P(\cdot/\|\cdot\|)|$ on $K^2\setminus\{0\}$
descends to $\bP^1(K)$ and in turn extends continuously to $\sP^1$
so that 
$\log S_P=\sum_{j=1}^{\deg P}\log[\cdot,z_j^P]_{\can}+M^\#(P)$
on $\sP^1$, which is rewritten as
$\log S_P-(\deg P)g=\sum_{j=1}^{\deg P}\Phi_g(\cdot,z_j^P)+M_g(P)$
on $\sP^1$. Integrating both sides against $\rd\mu^g$ over $\sP^1$,
by $U_{g,\mu^g}\equiv V_g$ on $\sP^1$, we have
the following {\itshape Jensen-type} formula
\begin{gather}
 M_g(P)=\int_{\sP^1}\left(\log S_P-(\deg P)g\right)\rd\mu^g-(\deg P)V_g.\label{eq:formulaheights}
\end{gather}

\subsection{A Lemma on global arithmetic}\label{sec:finiteness}
Let $k$ be a product formula field. 
The proof of the following
would not rely on
a usual argument based on a field extension of $k$.

\begin{lemma}\label{th:finite}
For every $P\in\bigcup_{d\in\bN}k[p_0,p_1]_d$,
$\sum_{v\in M_k}N_v\cdot M^\#(P)_v\in\bR_{\ge 0}$.
\end{lemma}

\begin{proof}
Let $(q_j^P)_{j=1}^{\deg P}$ be a sequence in $\overline{k}^2\setminus\{0\}$
giving a factorization \eqref{eq:factorization} of $P$, and
$L(P(1,\cdot))\in k\setminus\{0\}$ be the 
coefficient of the maximal degree term of $P(1,z)\in k[z]$. 
For each $j\in\{1,2,\ldots,\deg P\}$, setting $q_j^P=((q_j^P)_0,(q_j^P)_1)$, 
we have $L(P(1,\cdot))=(-1)^{\deg P-\deg_{\infty}P}
(\prod_{j:\pi(q_j^P)=\infty}(q_j^P)_1)
(\prod_{j:\pi(q_j^P)\neq\infty}(q_j^P)_0)$ since 
\begin{gather*}
 q_j^P=\begin{cases}
	(q_j^P)_0\cdot(1,\pi(q_j^P)) & \text{if }\pi(q_j^P)\neq\infty,\\
	(q_j^P)_1\cdot(0,1) & \text{if }\pi(q_j^P)=\infty.
       \end{cases}
\end{gather*}
Hence
we first have $\sum_{v\in M_k}N_v\cdot M^\#(P)_v\ge 
\sum_{v\in M_k}N_v\log|L(P(1,\cdot))|_v=0$,
where the final equality is by \eqref{eq:product}.
For each $i,j\in\bN\cup\{0\}$ satisfying $i+j=\deg P$,
 if the coefficient $a_{i,j}\in k$ of the
 expansion
$P(p_0,p_1)=\sum_{i,j\in\bN\cup\{0\}:i+j=\deg P}a_{i,j}p_0^ip_1^j$
in $k[p_0,p_1]_{\deg P}$ does not vanish, then by \eqref{eq:product},
there is a finite subset $E_{i,j}$ in $M_k$ 
such that for every $v\in M_k\setminus E_{i,j}$, $|a_{i,j}|_v=1$. 
Set 
$E_P:=\{\text{infinite places of }k\}
\cup\bigcup_{i,j\in\bN\cup\{0\}:i+j=\deg P\text{ and }a_{i,j}\neq 0}E_{i,j}$.
For every $v\in M_k\setminus E_P$, by the strong triangle inequality, we have
$|P(p_0,p_1)|_v\le\max\{\max\{|p_0|_v,|p_1|_v\}^{i+j}:i,j\in\bN\cup\{0\},i+j=\deg P\}
 =\|(p_0,p_1)\|_v^{\deg P}$ 
on $\bC_v^2$, so that $\log S_{P,v}\le 0$ on $\bP^1(\bC_v)$
and in turn on $\sP^1(\bC_v)$.
Set $g^0:=\{g^0_v:v\in M_k\}$ such that 
$g^0_v\equiv 0$ on $\sP^1(\bC_v)$ for every $v\in M_k$,
which is an adelic continuous weight.
For every finite $v\in M_k$, 
$\mu^{g^0}_v=\delta_{\cS_{\can,v}}$ on $\sP^1(\bC_v)$ and
$V_{g^0_v}=\log[\cS_{\can,v},\cS_{\can,v}]_{\can,v}
=0$, so that
by the Jensen-type formula \eqref{eq:formulaheights},
we have $M^\#(P)_v=M_{g^0_v}(P)=\log S_{P,v}(\cS_{\can,v})
$. 
Hence for every $v\in M_k\setminus E_P$, 
$M^\#(P)_v\le 0$, 
so that
$\sum_{v\in M_k}N_v\cdot M^\#(P)_v<\infty$
by $\# E_P<\infty$.
\end{proof}

\section{Regularization of discrete Radon measures whose supports are in $\bP^1$}
\label{sec:regularization}

Let $K$ be an algebraically closed field complete with respect to
a non-trivial absolute value $|\cdot|$.

\subsection{The small model metric $\sd$ and the Hsia kernel $|\cS-\cS'|_{\infty}$}\label{sec:Hsia}
The kernel function $[\cS,\cS']_{\can}$ is not necessarily
a metric on $\sP^1=\sP^1(K)$: indeed, for every $\cS\in\sP^1$,
$[\cS,\cS]_{\can}$ vanishes if and only if $\cS\in\bP^1=\bP^1(K)$. 
The {\itshape small model metric} $\sd$ on $\sP^1$ is a function
\begin{gather}
 \sd(\cS,\cS'):=[\cS,\cS']_{\can}-\frac{[\cS,\cS]_{\can}+[\cS',\cS']_{\can}}{2}
 \quad\text{on }\sP^1\times\sP^1,\label{eq:small}
\end{gather}
which extends the normalized chordal metric $[z,w]$ on $\bP^1$
(but this $\sd$ does not induce the topology of $\sP^1$;
see \cite[\S2.7]{BR10}, \cite[\S4.7]{FR06} for more details).
On the other hand, the {\itshape Hsia kernel} $|\cS-\cS'|_{\infty}$
on the {\itshape Berkovich} affine line
$\sA^1=\sA^1(K)=\sP^1\setminus\{\infty\}$ is the function
\begin{gather}
 |\cS-\cS'|_{\infty}:=
 [\cS,\cS']_{\can}\cdot[\cS,\infty]_{\can}^{-1}\cdot[\cS',\infty]_{\can}^{-1}
 \quad\text{on }\sA^1\times\sA^1,\label{eq:Hsiaaffine}
\end{gather}
although the difference $\cS-\cS'$ itself is not defined unless both $\cS,\cS'\in K$ (for more details, see \cite[Chapter 4]{BR10}).
The kernel $|\cS-\cS'|_{\infty}$ is
the unique (jointly) upper semicontinuous and
separately continuous extension of  the function $|z-w|$ on $K\times K$
to $\sA^1\times\sA^1$.

\subsection{A little description of $\sP^1$ for non-archimedean $K$}\label{sec:tree}
Suppose that $K$ is non-archimedean.
A subset $B$ in $K$ is called a ($K$-closed) {\itshape disk} in $K$
if $B=\{z\in K:|z-a|\le r\}$ for some $a\in K$ and some {\itshape radius} $r\ge 0$.
By the strong triangle inequality, 
{\itshape two disks in $K$ either nest or are disjoint}.
This alternative extends to any two decreasing infinite sequences 
of disks in $K$ so that they either {\itshape infinitely nest}
or {\itshape are eventually disjoint}, and induces a so called {\itshape cofinal} 
equivalence relation among them.
\begin{example}
 Instead of
 giving
 a formal definition of the cofinal equivalence class $\cS$
 of a decreasing infinite sequence $(B_n)$ of disks in $K$,
 let us be practical: each $z\in K$ is
 regarded as
 the cofinal equivalence class of
 the constant
 sequence $(B_n)$ of the disks $B_n\equiv\{z\}$ in $K$ (of radii $\equiv 0$).
 More generally, for every cofinal equivalence class $\cS$
 of a decreasing infinite sequence $(B_n)$ of disks
 in $K$, $B_{\cS}:=\bigcap_{n\in\bN}B_n$
 is independent of choices of the representative $(B_n)$ of $\cS$,
 and if $B_{\cS}\neq\emptyset$, then
 $B_{\cS}$ is still a disk in $K$ and the $\cS$ is
 represented by
 the constant sequence $(\tilde{B}_n)$ of the disks $\tilde{B}_n\equiv B_{\cS}$ in $K$.
\end{example}

As a set,
the set of all cofinal equivalence classes $\cS$ of
decreasing infinite sequences 
$(B_n)$ of disks 
in $K$ and in addition $\infty\in\bP^1$ 
is nothing but $\sP^1$ (\cite[p.\ 17]{Berkovichbook}.
See also \cite[\S 2]{BR10}, \cite[\S 3]{FR06}, \cite[\S 6.1]{Benedetto10}):
for example, the {\itshape canonical $($or Gauss$)$} point $\cS_{\can}$ in $\sP^1$
is represented by the ring of $K$-integers $\cO_K:=\{z\in K:|z|\le 1\}$,
which is a disk in $K$.
The above alternative induces
a partial ordering $\succeq$ on $\sP^1$ so that
for every $\cS,\cS'\in\sP^1$ satisfying $B_{\cS},B_{\cS'}\neq\emptyset$, 
$\cS\succeq\cS'$ if and only if $B_{\cS}\supset B_{\cS'}$
(the description is a little complicate
when one of $B_{\cS},B_{\cS'}$ equals $\emptyset$). 
For every $\cS,\cS'\in\sP^1$ satisfying $\cS\succeq\cS'$, 
the {\itshape segment} between $\cS$ and $\cS'$ in $\sP^1$
is the set of all points $\cS''\in\sP^1$ satisfying $\cS\succeq\cS''\succeq\cS'$,
which can be equipped with either the ordering induced by
$\succeq$ on $\sP^1$ or its opposite one. All those (oriented) segments
make $\sP^1$ a {\itshape tree} in the sense of Jonsson
\cite[\S 2, Definition 2.2]{Jonsson15}.

For each $\cS\in\sP^1\setminus\{\infty\}$
represented by $(B_n)$, set 
\begin{gather*}
 \diam\cS:=\lim_{n\to\infty}\diam B_n(=\diam B_{\cS}
 \text{ if }\cB_{\cS}\neq\emptyset),
\end{gather*}
where $\diam B$ denotes the diameter of a disk $B$ in $K$
with respect to $|\cdot|$:
by convention, for $\cS=\infty$, we set $B_{\infty}:=K$ and $\diam\infty:=+\infty$.
The {\itshape hyperbolic space} is
$\sH^1=\sH^1(K):=\sP^1\setminus\bP^1=\{\cS\in\sP^1:\diam\cS\in(0,+\infty)\}$.
The {\itshape big model $($or hyperbolic$)$ metric} $\rho$ on $\sH^1$ is
a path metric on $\sH^1$ (but does not induce the relative topology of
$\sH^1$ induced by $\sP^1$) so that
for every $\cS,\cS'\in\sH^1$ satisfying $\cS\succeq\cS'$,
\begin{gather}
 \rho(\cS,\cS')=\log(\diam\cS/\diam\cS')\label{eq:hypdist}
\end{gather}
(see, e.g., \cite[\S 2.7]{BR10}).
In terms of $\rho$, the generalized Hsia kernel $[\cS,\cS']_{\can}$
with respect to $\cS_{\can}$ is interpreted as a Gromov product
\begin{gather}
 \log[\cS,\cS']_{\can}=-\rho(\cS'',\cS_{\can})\quad\text{on }\sH^1\times\sH^1,\label{eq:Gromov}
\end{gather}
where $\cS''$ is the unique point in $\sH^1$ lying between $\cS$ and $\cS'$,
between $\cS'$ and $\cS_{\can}$, and between $\cS_{\can}$ and $\cS$
(cf.\ \cite[\S 3.4]{FR06}). Similarly, for every $\cS,\cS'\in\sA^1$,
\begin{gather}
 |\cS-\cS'|_{\infty}=\diam\cS'',\label{eq:affineGromov}
\end{gather}
where $\cS''$ is the smallest point in $\sA^1$
satisfying both $\cS''\succeq\cS$ and $\cS''\succeq\cS'$
with respect to the partial ordering $\succeq$ on $\sP^1$. 

For every $\epsilon>0$, a continuous mapping
\begin{gather*}
 \pi_{\epsilon}:\sA^1\to\sA^1
\end{gather*}
is defined so that
for each $\cS\in\sA^1$, $\pi_{\epsilon}(\cS):=\cS''\in
\sH^1_\epsilon:=\{\cS\in\sP^1:\diam\cS\in[\epsilon,+\infty)\}$,
where $\cS''$
is the unique point between $\infty$ and $\cS$
satisfying $\diam\cS''=\max\{\epsilon,\diam\cS\}$ (see \cite[\S4.6]{FR06}
for the details).

\subsection{Regularization on $\sP^1$}

 When $K$ is archimedean, fix a non-negative
 smooth decreasing function 
 $\xi:[0,\infty)\to[0,1]$ such that $\supp\xi\subset[0,1]$ and
 that $\int_0^{\infty}\xi(x)\rd x=1$, and
 set $\xi_{\epsilon}(x):=\xi(x/\epsilon)/\epsilon$ on $[0,+\infty)$ for
 each $\epsilon>0$. For every $z\in K$ and every $\epsilon>0$,
 the {\itshape $\epsilon$-regularization} $[z]_{\epsilon}$ of $\delta_z$ is 
 the convolution $\xi_{\epsilon}*\delta_z$ on $\bP^1$, 
 i.e., for any continuous test function $\phi$ on $\bP^1$,
\begin{gather*}
(\xi_{\epsilon}*\delta_z)(\phi)=\int_0^{\epsilon}\xi_{\epsilon}(r)\rd r\int_0^{2\pi}\phi(z+re^{i\theta})\frac{\rd\theta}{2\pi}.
\end{gather*}
 When $K$ is non-archimedean,
 for every $z\in K$ and every $\epsilon>0$,
 the {\itshape $\epsilon$-regularization} $[z]_{\epsilon}$ of $\delta_z$ 
 is defined by 
 $[z]_{\epsilon}:=(\pi_{\epsilon})_*\delta_z=\delta_{\pi_{\epsilon}(z)}$
 on $\sP^1$ (\cite[p.\ 343]{FR06}).
 In both cases,
$[z]_{\epsilon}$ is a probability Radon measure on $\sP^1$,
the {\itshape chordal} potential
$\sP^1\ni \cS\mapsto \int_{\sP^1}\log[\cS,\cS']_{\can}\rd[z]_{\epsilon}(\cS')$
of $[z]_{\epsilon}$ is a continuous function on $\sP^1$,
and for every $z,w\in K$ and every $\epsilon>0$, the estimate
 \begin{gather}
 \int_{\sA^1\times\sA^1}\log|\cS-\cS'|_{\infty}\rd([z]_{\epsilon}\times [w]_{\epsilon})(\cS,\cS')
 \ge
 \begin{cases}
\log|z-w| & \text{if }z\neq w\\
 C_{\operatorname{abs}}+\log\epsilon & \text{if }z=w
 \end{cases}\label{eq:regularizationaffine}
 \end{gather}
 holds,
 where the constant $C_{\operatorname{abs}}\le 0$ is an absolute constant
 ({\cite[Lemmes 2.10, 4.11, and their proofs]{FR06}}).

Let us extend the $\epsilon$-regularization $[\cdot]_{\epsilon}$ and
the estimate \eqref{eq:regularizationaffine} to $\sP^1$.
Set $\iota(z):=1/z\in\PGL(2,k)$, which extends to an automorphism on $\sP^1$
(cf.\ Fact \ref{eq:action}), so that $\iota^2=\Id$ on $\sP^1$ and that
$[\iota(\cS),\iota(\cS')]_{\can}=[\cS,\cS']_{\can}$
on $\sP^1\times\sP^1$. For every $\epsilon>0$, set 
$[\infty]_{\epsilon}:=\iota_*[0]_{\epsilon}$.

For every $z\in\bP^1$ and every $\epsilon>0$,
\begin{gather}
\supp[z]_{\epsilon}
\subset\{\cS\in\sP^1:\sd(\cS,z)\le\epsilon\},\label{eq:supportsmall}
\end{gather}
which immediately follows from the definitions of
$|\cS-\cS'|_{\infty}$ (and \eqref{eq:affineGromov}),
$\sd$, and $[z]_{\epsilon}$ when $z\in K$, and from \eqref{eq:supportsmall}
applied to $z=0$ and the invariance of $\sd$ under $\iota$ when $z=\infty$. Moreover,
for every $z\in K$ and every $\epsilon>0$,
\begin{gather}
 \sup_{\cS\in[z]_{\epsilon}}|\log[\cS,\infty]_{\can}-\log[z,\infty]|\le\epsilon\label{eq:moduluschordal}
\end{gather}
by a direct computation of $\log[\cdot,\infty]_{\can}-\log[z,\infty]$
on $K$, the definition of $[z]_\epsilon$,
and the density of $K$ in $\sA^1$.

\begin{lemma}
Let $g$ be a continuous weight on $\sP^1$ 
having a modulus of continuity $\eta$ on $(\sP^1,\sd)$.
Then for every $\epsilon>0$
and every $z,w\in\bP^1$,
\begin{multline}
 \int_{\sP^1\times\sP^1}\Phi_g\rd([z]_{\epsilon}\times[w]_{\epsilon})\\
\ge
\begin{cases}
 \Phi_g(z,w)-2\epsilon-2\eta(\epsilon) & \text{if }z\neq w,\\
 C_{\operatorname{abs}}+\log\epsilon-2\epsilon+2\log[z,\infty]
-2\eta(\epsilon)-2g(z) & \text{if }z=w\in K,\\
 C_{\operatorname{abs}}+\log\epsilon-2\epsilon-2\eta(\epsilon)
 -2g(\infty) & \text{if }z=w=\infty.
 \end{cases}\label{eq:regularization} 
\end{multline}  
\end{lemma}

\begin{proof}
Since by definition $\Phi_g(\cS,\cS')=\log[\cS,\cS']_{\can}-g(\cS)-g(\cS')$
on $\sP^1\times\sP^1$, 
we can assume $g\equiv 0$ (and $\eta\equiv 0$) on $\sP^1$ without loss of
generality.
For every $z,w\in K$, 
by the definition \eqref{eq:Hsiaaffine} of $|\cS-\cS'|_{\infty}$
and \eqref{eq:moduluschordal},
\begin{align*}
 &\int_{\sP^1\times\sP^1}\log[\cS,\cS']_{\can}\rd([z]_{\epsilon}\times[w]_{\epsilon})(\cS,\cS')\\
\ge&\int_{\sA^1\times\sA^1}\log|\cS-\cS'|_{\infty}\rd([z]_{\epsilon}\times [w]_{\epsilon})(\cS,\cS')-2\epsilon+\log[z,\infty]+\log[w,\infty], 
\end{align*} 
which with the estimate \eqref{eq:regularizationaffine} 
yields \eqref{eq:regularization} (for $g\equiv\eta\equiv 0$) in this case.
The estimate \eqref{eq:regularization} (for $g\equiv\eta\equiv 0$)
in the case $z=w=\infty$
follows from $[\infty]_\epsilon=\iota_*[0]_\epsilon$,
$[\iota(\cS),\iota(\cS')]_{\can}=[\cS,\cS']_{\can}$, and
the estimate \eqref{eq:regularization}
for $z=w=0$.

It remains the case that $z=\infty$ and $w\in K$ (so $z\neq w$).
If $K$ is non-archimedean, then for every $w\in K$
and every $\epsilon>0$, by $[\infty]_\epsilon=\iota_*[0]_\epsilon$,
$[\iota(\cS),\iota(\cS')]_{\can}=[\cS,\cS']_{\can}$, and the interpretation
\eqref{eq:Gromov} of $[\cS,\cS']_{\can}$, we have
\begin{multline*} 
 \int_{\sP^1\times\sP^1}\log[\cS,\cS']_{\can}\rd([\infty]_{\epsilon}\times[w]_{\epsilon})(\cS,\cS')\\
=
\int_{\sP^1\times\sP^1}\log[\cS,\cS']_{\can}\rd([0]_{\epsilon}\times\iota_*[w]_{\epsilon})(\cS,\cS')
=\log[\pi_{\epsilon}(0),\iota(\pi_{\epsilon}(w))]_{\can}\\
 \ge\log[0,1/w]=\log[\infty,w]
 \ge \log[\infty,w]-2\epsilon,
\end{multline*}
which implies the estimate
\eqref{eq:regularization} (for $g\equiv\eta\equiv 0$)
in the case $z=\infty$ and $w\in K$ when $K$ is non-archimedean.
If $K$ is archimedean, then for every $w\in K$ and every $r,r'>0$,
\begin{multline*}
\int_0^{2\pi}\frac{\rd\theta}{2\pi}\int_0^{2\pi}\log\left|(0+re^{i\theta})-\frac{1}{w+r'e^{i\phi}}\right|\frac{\rd\phi}{2\pi}\\
=\int_0^{2\pi}\max\left\{-\log|w+r'e^{i\phi}|,\log r\right\}\frac{\rd\phi}{2\pi}
\ge
-\int_0^{2\pi}\log|(w+r'e^{i\phi})-0|\frac{\rd\phi}{2\pi},
\end{multline*}
so that for every $w\in K\cong\sA^1$ and every $\epsilon>0$,
\begin{multline*}
 \int_{\sA^1\times\sA^1}\log|\cS-\cS'|_{\infty}\rd([0]_{\epsilon}\times\iota_*[w]_{\epsilon})(\cS,\cS')\\
=\int_{\sA^1\times\sA^1}\log|\cS-\iota(\cS')|_{\infty}\rd([0]_{\epsilon}\times[w]_{\epsilon})(\cS,\cS')
\ge-\int_{\sA^1}\log|\cS'-0|_{\infty}\rd[w]_{\epsilon}(\cS'). 
\end{multline*}
On the other hand,
for every $w\in K$ and every $\epsilon>0$,
by the definition \eqref{eq:chordaldist}
of the chordal metric $[z,w]$ on $\bP^1\cong\sP^1$ (and $[0,\infty]=1$), 
\begin{multline*}
 \int_{\sP^1}\log[\cS',\infty]_{\can}\rd(\iota_*[w]_{\epsilon})(\cS')
=\int_{\sP^1}\log[\cS',0]_{\can}\rd[w]_{\epsilon}(\cS')\\
=\int_{\sA^1}\log|\cS'-0|_{\infty}\rd[w]_{\epsilon}(\cS')
+\int_{\sP^1}\log[\cS',\infty]_{\can}\rd[w]_{\epsilon}(\cS').
\end{multline*}
By these computations
and \eqref{eq:moduluschordal},
for every $w\in K$ and every $\epsilon>0$,
\begin{align*}
 &\int_{\sP^1\times\sP^1}\log[\cS,\cS']_{\can}\rd([\infty]_{\epsilon}\times[w]_{\epsilon})(\cS,\cS')\\
=&
\int_{\sP^1\times\sP^1}\log[\cS,\cS']_{\can}\rd([0]_{\epsilon}\times\iota_*[w]_{\epsilon})(\cS,\cS')\\
\ge&
\int_{\sP^1}\log[\cS,\infty]_{\can}\rd[0]_{\epsilon}(\cS)
+\int_{\sP^1}\log[\cS',\infty]_{\can}\rd[w]_{\epsilon}(\cS')\\
\ge&\log[0,\infty]+\log[w,\infty]-2\epsilon=\log[w,\infty]-2\epsilon,
\end{align*}
which implies the estimate
\eqref{eq:regularization} (for $g\equiv\eta\equiv 0$)
in the case $z=\infty$ and $w\in K$ when $K$ is archimedean.
\end{proof}

\section{The negativity of regularized Fekete sums and a Cauchy-Schwarz inequality}
\label{sec:CS}

Let $K$ be an algebraically closed field that is complete with respect to
a non-trivial absolute value $|\cdot|$.
For every $\epsilon>0$ and
every discrete measure $\nu$ on $\sP^1=\sP^1(K)$ whose support is in $\bP^1=\bP^1(K)$,
the {\itshape $\epsilon$-regularization} of $\nu$ is
\begin{gather*}
 \nu_{\epsilon}:=\sum_{w\in\supp\nu}\nu(\{w\})[w]_{\epsilon}\quad\text{on }\sP^1
\end{gather*}
and, for every continuous weight $g$ on $\sP^1$,
let us call $(\nu_\epsilon,\nu_\epsilon)_g$
the {\itshape $\epsilon$-regularized $g$-Fekete sum}
with respect to $\nu$.

\subsection{$C^1$-regularity and the Dirichlet norm}
Recall \S \ref{sec:tree} (a little description on $\sP^1$).
When $K$ is non-archimedean, 
 a function $\phi$ on $\sP^1=\sP^1(K)$ is {\itshape in} $C^1(\sP^1)$ if 
 (i) $\phi$ is continuous on $\sP^1$ and
 locally constant except for a union $\mathcal{T}$
 of at most finitely many segments in $\sH^1=\sH^1(K)$, 
 which are oriented by the partial ordering $\succeq$ on $\sP^1$
 and (ii) the derivative $\phi'$ 
 with respect to the length parameter 
 induced by the hyperbolic metric $\rho$ on each segment 
 in $\mathcal{T}$ exists and is continuous on $\mathcal{T}$. 
 The {\itshape Dirichlet norm} of 
 $\phi\in C^1(\sP^1)$ 
 is defined by $\langle\phi,\phi\rangle^{1/2}:=(\int_{\mathcal{T}}(\phi')^2\rd\rho)^{1/2}$, where $\rd\rho$ is the $1$-dimensional Hausdorff measure on $\sH^1$
 with respect to $\rho$ (for more details, see \cite[\S 5.5]{FR06}). 
 When $K$ is archimedean, the $C^1$-regularity and the Dirichlet norm
 of a function $\phi$ on $\sP^1\cong\bP^1$ is defined with respect to 
 the complex (or differentiable) structure of $\bP^1$.
 For completeness, we include a proof of the following.

\begin{proposition}\label{th:Lipschitz}
Every $\phi$ in $C^1(\sP^1)$ is Lipschitz continuous on $(\sP^1,\sd)$.
\end{proposition}

\begin{proof}
When $K$ is archimedean, it is obvious. Suppose that $K$ is non-archimedean and
let $\phi\in C^1(\sP^1)$. By definition, $\phi$ is locally constant on $\sP^1$
except for a union $\mathcal{T}$ of at most finitely many segments in $\sH^1$
and is Lipschitz continuous on $\mathcal{T}$ with respect to $\rho$.
The set $\mathcal{T}$ is compact in $(\sH^1,\rho)$, and
for every $\cS,\cS'\in\sH^1$,
by the definition \eqref{eq:small} of $\sd$,
\eqref{eq:Gromov}, and \eqref{eq:hypdist},
if $\cS_{\can}\succeq\cS\succeq\cS'$, then
\begin{multline*}
  \sd(\cS,\cS')=\diam\cS-\frac{\diam\cS+\diam\cS'}{2}\\
=\frac{\diam\cS-\diam\cS'}{2}
\ge\frac{\diam\cS'}{2}\rho(\cS,\cS'),
\end{multline*} 
and similarly, if $\cS_{\can}\preceq\cS\preceq\cS'$, then
$\sd(\cS,\cS')\ge\rho(\cS,\cS')/(2\diam\cS')$.
Hence we conclude that $\phi$ is also Lipschitz continuous on $\mathcal{T}$ 
with respect to $\sd$, and in turn on the whole $\sP^1$ with respect to $\sd$.
\end{proof}
 
The Lipschitz constant of a Lipschitz continuous function $\phi$ 
on $(\sP^1,\sd)$ is denoted by $\Lip(\phi)$.

\begin{remark}\label{th:dense}
 When $K$ is archimedean (so $\sP^1\cong\bP^1$),
 for every $\phi\in C^1(\bP^1)$,
$\langle\phi,\phi\rangle^{1/2}\le\Lip(\phi)$. Moreover,
every Lipschitz continuous function $\phi$ on $(\bP^1,[z,w])$
is approximated by functions in $C^1(\bP^1)$ in the Lipschitz norm.
\end{remark}

\subsection{The negativity of $(\nu_\epsilon,\nu_\epsilon)_g$ and a Cauchy-Schwarz inequality}\label{sec:negative}
For every Radon measure $\mu$ on $\sP^1$
satisfying that $\mu(\sP^1)=0$, if the chordal potential 
$\cS\mapsto \int_{\sP^1}\log[\cS,\cS']_{\can}\rd\mu(\cS')$ 
of $\mu$ is continuous on $\sP^1$, then we have
the {\itshape positivity} $\int_{\sP^1\times\sP^1}
(-\log|\cS-\cS'|_{\infty})\rd(\mu\times\mu)(\cS,\cS')\ge 0$
(see \cite[\S 2.5 and \S 4.5]{FR06}) and, moreover,
the {\itshape Cauchy-Schwarz inequality}
\begin{gather}
\left|\int_{\sP^1}\phi\rd\mu\right|^2
 \le\langle\phi,\phi\rangle\cdot
 \int_{\sP^1\times\sP^1}(-\log|\cS-\cS'|_{\infty})
 \rd(\mu\times\mu)(\cS,\cS')\label{eq:CS}
\end{gather}
for every test function $\phi\in C^1(\sP^1)$ (see \cite[(32) and (33)]{FR06}).

In particular, for every $\epsilon>0$, every normalized weight $g$ on $\sP^1$,
every test function $\phi\in C^1(\sP^1)$, and
every discrete measure $\nu$ on $\sP^1$ whose support is in $\bP^1$,
the computation
 \begin{multline*}
0\le \int_{\sP^1\times\sP^1}(-\log|\cS-\cS'|_{\infty})
 \rd((\nu_{\epsilon}-(\nu(\sP^1))\mu^g)\times(\nu_{\epsilon}-(\nu(\sP^1))\mu^g))(\cS,\cS')\\
  =\int_{\sP^1\times\sP^1}(-\Phi_g)\rd((\nu_{\epsilon}-(\nu(\sP^1))\mu^g)\times(\nu_{\epsilon}-(\nu(\sP^1))\mu^g))
  =-(\nu_{\epsilon},\nu_{\epsilon})_g
\end{multline*}
(recalling $\Phi_{g,\mu^g}\equiv 0$ on $\sP^1$)
yields not only the {\itshape negativity}
$(\nu_{\epsilon},\nu_{\epsilon})_g\le 0$
but,
with the Cauchy-Schwarz inequality \eqref{eq:CS} 
and the triangle inequality, also the estimate
 \begin{multline}
  \left|\int_{\sP^1}\phi\rd\left(\nu-(\nu(\sP^1))\mu^g\right)\right|
=\left|\int_{\sP^1}\phi\rd\left((\nu-\nu_\epsilon)+(\nu_\epsilon-(\deg\nu)\mu^g)\right)\right|\\
  \le (\deg\nu)\Lip(\phi)\epsilon+
 \langle\phi,\phi\rangle^{1/2}\cdot(-(\nu_{\epsilon},\nu_{\epsilon})_g)^{1/2}.\label{eq:schwarz} 
 \end{multline}

\section{Computations of Fekete sums $(\cZ,\cZ)_g$}\label{sec:algebraic}

Let $k$ be a field.
For a $k$-algebraic zeros divisor $\cZ$ on $\bP^1(\overline{k})$, set
\begin{gather*}
 D^*(\cZ|\overline{k})
 :=\prod_{w\in\supp\cZ\setminus\{\infty\}}\prod_{w'\in\supp\cZ\setminus\{w,\infty\}}(w-w')^{(\ord_{w}\cZ)(\ord_{w'}\cZ)}\in\overline{k}\setminus\{0\},
\end{gather*}
which is indeed in $k\setminus\{0\}$ (by Theorem \ref{th:discriminant}).
For every $P\in\bigcup_{d\in\bN}k[p_0,p_1]_d$, let $L(P(1,\cdot))\in k\setminus\{0\}$
be the coefficient of the maximal degree term of $P(1,z)\in k[z]$
(appearing in \S\ref{sec:finiteness}).


\begin{lemma}\label{th:indep}
Let $k$ be a field.
Let $\cZ$ be a $k$-algebraic zeros divisor on $\bP^1(\overline{k})$
represented by $P\in\bigcup_{d\in\bN}k[p_0,p_1]_d$ and
$(q_j^P)_{j=1}^{\deg P}$
a sequence in $\overline{k}^2\setminus\{0\}$ 
giving a factorization \eqref{eq:factorization} of $P$, and
for each $j\in\{1,2,\ldots,\deg P\}$, 
set $q_j^P=((q_j^P)_0,(q_j^P)_0)$ and $z_j:=\pi(q_j^P)\in\bP^1(\overline{k})$.
If $(q_j^P)_{j=1}^{\deg P}$ is normalized with respect to a distinguished zero 
$w_0\in\bP^1(\overline{k})$ of $P$
so that for each $j\in\{1,2,\ldots,\deg P\}$,
\begin{gather}\label{eq:normalization}
\begin{cases}
 (q_j^P)_0=1 & \text{if }z_j\not\in\{w_0,\infty\},\\ 
 (q_j^P)_1=1 & \text{if }w_0\neq z_j=\infty,
\end{cases}
\end{gather}
then
\begin{gather}
L(P(1,\cdot))=(-1)^{\deg P-\deg_\infty P}\times
\begin{cases}
 \prod_{j:z_j=w_0}(q_j^P)_0 & \text{if }w_0\neq\infty,\\
 \prod_{j:z_j=w_0}(q_j^P)_1 & \text{if }w_0=\infty,
\end{cases}\label{eq:prodequi}
\end{gather}
and
\begin{multline}
 \prod_{j=1}^{\deg P}\prod_{i:z_i\neq z_j}(q_i^P\wedge q_j^P)\\
 =(-1)^{\deg_{\infty}P(\deg P-\deg_{\infty}P)}\cdot
L(P(1,\cdot))^{2(\deg P-\deg_{w_0}P)}
 \cdot D^*(\cZ|\overline{k}).\label{eq:planner}
\end{multline}
\end{lemma}

\begin{proof}
 Without normalizing the sequence $(q_j^P)_{j=1}^{\deg P}$, by a direct
 computation,
\begin{multline}
 \prod_{j=1}^{\deg P}\prod_{i:z_i\neq z_j}(q_i^P\wedge q_j^P)\\
=\left(\prod_{j:z_j=\infty}\prod_{i:z_i\neq\infty}((q_i^P)_0(q_j^P)_1)\right)
 \times\left(\prod_{j:z_j\neq\infty}\prod_{i:z_i=\infty}(-(q_i^P)_1(q_j^P)_0)\right)\times\\
 \times\prod_{j:z_j\neq\infty}\prod_{i:z_i\not\in\{z_j,\infty\}}
((q_i^P)_0(q_j^P)_0(z_j-z_i))
=(-1)^{\deg_\infty P(\deg P-\deg_\infty P)}\times\\
\times\left(\prod_{j:z_j=\infty}\left((q_j^P)_1^{\deg P-\deg_{\infty}P}\cdot
\prod_{i:z_i\neq\infty}(q_i^P)_0\right)\right)^2\times\\
\times\left(\prod_{j:z_j\neq\infty}
\left((q_j^P)_0^{\deg P-\deg_\infty P-\deg_{z_j}P}
\cdot\prod_{i:z_i\not\in\{z_j,\infty\}}(q_i^P)_0\right)\right)
\cdot D^*(\cZ|\overline{k}).\label{eq:separate}
\end{multline} 
Let us normalize $(q_j^P)$ so that 
the normalization \eqref{eq:normalization} holds
with respect to a distinguished zero $w_0\in\bP^1(\overline{k})$ of $P$.
 Then \eqref{eq:prodequi} follows from
$L(P(1,\cdot))=
(-1)^{\deg P-\deg_\infty P}\cdot(\prod_{j:z_j=\infty}(q_j^P)_1)
(\prod_{j:z_j\neq\infty}(q_j^P)_0)$
and the normalization \eqref{eq:normalization}. 
Let us show \eqref{eq:planner}. 
If $w_0=\infty$, then under the normalization \eqref{eq:normalization},
the equality \eqref{eq:separate} yields
\begin{multline*}
\prod_{j=1}^{\deg P}\prod_{i:z_i\neq z_j}(q_i^P\wedge q_j^P)\\
=(-1)^{\deg_\infty P(\deg P-\deg_\infty P)}
\cdot
\left(\prod_{j:z_j=\infty}(q_j^P)_1\right)^{2(\deg P-\deg_{\infty}P)}
\cdot 1\cdot D^*(\cZ|\overline{k}),
\end{multline*}
which with \eqref{eq:prodequi} implies \eqref{eq:planner} when $w_0=\infty$. 
If $w_0\neq\infty$, then under the normalization \eqref{eq:normalization},
the equality \eqref{eq:separate} yields
\begin{multline*}
\prod_{j=1}^{\deg P}\prod_{i:z_i\neq z_j}(q_i^P\wedge q_j^P)
=(-1)^{\deg_{\infty}P(\deg P-\deg_{\infty}P)}
\cdot\left(\prod_{i:z_i=w_0}(q_i^P)_0\right)^{2\deg_{\infty}P}\times\\
\times
\left(\prod_{j:z_j=w_0}\left((q_j^P)_0^{\deg P-\deg_{\infty}P-\deg_{z_j}P}\cdot 1\right)\right)\cdot\left(\prod_{j:z_j\not\in\{w_0,\infty\}}
(1\cdot\prod_{i:z_i=w_0}(q_i^P)_0)\right)\times\\
\times D^*(\cZ|\overline{k})\\
=(-1)^{\deg_{\infty}P(\deg P-\deg_{\infty}P)}
\cdot\left(\prod_{i:z_i=w_0}(q_i^P)_0\right)
^{2\deg_{\infty}P+2(\deg P-\deg_{\infty}P-\deg_{w_0}P)}\times\\
\times D^*(\cZ|\overline{k}),
\end{multline*}  
which with \eqref{eq:prodequi}
implies \eqref{eq:planner} when $w_0\neq\infty$.
\end{proof}

\begin{lemma}[local computation]\label{th:discrepancy} 
Let $k$ be a field and $K$ an algebraic and metric augmentation of $k$
$($see \S $\ref{sec:local})$.
For every continuous weight $g$ on $\sP^1=\sP^1(K)$ and
every $k$-algebraic zeros divisor $\cZ$ on $\bP^1(\overline{k})$
represented by $P\in\bigcup_{d\in\bN}k[p_0,p_1]_d$, 
\begin{multline}
 (\cZ,\cZ)_g\\
 +2\cdot\sum_{w\in\supp\cZ\setminus\{\infty\}}(\ord_w\cZ)^2
\log[w,\infty]-2\cdot\sum_{w\in\supp\cZ}(\ord_w\cZ)^2g(w)\\
=2(\deg\cZ)\log|L(P(1,\cdot))|+\log|D^*(\cZ|\overline{k})|-2(\deg\cZ)M_g(P).\label{eq:discrepancystrong}
\end{multline}
\end{lemma}

\begin{proof}
Let $\cZ$ be a $k$-algebraic zeros divisor on $\bP^1(\overline{k})$
represented by $P\in\bigcup_{d\in\bN}k[p_0,p_1]_d$ and
$(q_j^P)_{j=1}^{\deg P}$ be
a sequence
in $\overline{k}^2\setminus\{0\}$ giving a factorization 
\eqref{eq:factorization} 
of $P$ and satisfying the normalization \eqref{eq:normalization}
with respect to a distinguished zero $w_0\in\bP^1(\overline{k})$ of $P$, and 
set $z_j:=\pi(q_j^P)\in\bP^1(\overline{k})$ for each $j\in\{1,2,\ldots,\deg P\}$.
Since by definition
$\Phi_g(z,z')=\log[z,z']-g(z)-g(z')$ on $\bP^1(K)\times\bP^1(K)$,
we have
$(\cZ,\cZ)_g
=\log(\prod_{j=1}^{\deg P}\prod_{i:z_i\neq z_j}|q_i^P\wedge q_j^P|)
-2\cdot\sum_{j=1}^{\deg P}\sum_{i:z_i\neq z_j}(g(z_i)+\log\|q_i^P\|)$;
by \eqref{eq:planner},
$\log(\prod_{j=1}^{\deg P}\prod_{i:z_i\neq z_j}|q_i^P\wedge q_j^P|)
 =2(\deg P-\deg_{w_0}P)\log|L(P(1,\cdot))|+\log|D^*(\cZ|\overline{k})|$,
and we also have
\begin{align*}
 &\sum_{j=1}^{\deg P}\sum_{i:z_i\neq z_j}(g(z_i)+\log\|q_i^P\|)\\
=&\sum_{j=1}^{\deg P}\sum_{i=1}^{\deg P}(g(z_i)+\log\|q_i^P\|)
-\sum_{j=1}^{\deg P}\sum_{i:z_i=z_j}(g(z_i)+\log\|q_i^P\|)\\
=&(\deg P)M_g(P)-\sum_{j=1}^{\deg P}(\deg_{z_j}P)g(z_j)
-\sum_{j=1}^{\deg P}\sum_{i:z_i=z_j}\log\|q_i^P\|,
\end{align*}
where the final equality is by the definition \eqref{eq:Mahler} of $M_g(P)$.
Hence
\begin{multline*}
 (\cZ,\cZ)_g
=2(\deg P)\log|L(P(1,\cdot))|+\log|D^*(\cZ|\overline{k})|\\
-2(\deg P)M_g(P)+2\sum_{w\in\supp\cZ}(\ord_w\cZ)^2g(w)\\
-2\left((\deg_{w_0}P)\log|L(P(1,\cdot))|
-\sum_{j=1}^{\deg P}\sum_{i:z_i=z_j}\log\|q_i^P\|\right).
\end{multline*}
For each $j\in\{1,2,\ldots,\deg P\}$, also set $q_j^P=((q_j^P)_0,(q_j^P)_0)$.
If $\infty\not\in\supp\cZ$, then $w_0\neq\infty$, and
by the normalization \eqref{eq:normalization} and 
the equality \eqref{eq:prodequi},
\begin{multline*}
 (\deg_{w_0}P)\log|L(P(1,\cdot))|-\sum_{j=1}^{\deg P}\sum_{i:z_i=z_j}\log\|q_i^P\|\\
=-\sum_{j=1}^{\deg P}\sum_{i:z_i=z_j}(\log\|q_i^P\|-\log|(q_i^P)_0|)
=\sum_{j=1}^{\deg P}\sum_{i:z_i=z_j}\log[z_i,\infty]\\
 =\sum_{w\in\supp\cZ}(\ord_w\cZ)^2\log[w,\infty]
 =\sum_{w\in\supp\cZ\setminus\{\infty\}}(\ord_w\cZ)^2\log[w,\infty].
\end{multline*}
If $\infty\in\supp\cZ$, then 
we can set $w_0=\infty$,
and by the normalization \eqref{eq:normalization} and
the equality \eqref{eq:prodequi} (and $q_i^P=(q_i^P)_1\cdot (0,1)$ when $z_i=\infty$),
\begin{multline*}
 (\deg_{w_0}P)\log|L(P(1,\cdot))|
-\sum_{j=1}^{\deg P}\sum_{i:z_i=z_j}\log\|q_i^P\|\\
=
-\sum_{j:z_j=\infty}\sum_{i:z_i=z_j}(\log\|q_i^P\|-\log|(q_i^P)_1|)
-\sum_{j:z_j\neq\infty}\sum_{i:z_i=z_j}(\log\|q_i^P\|-\log|(q_i^P)_0|)\\
=
\sum_{j:z_j\neq\infty}\sum_{i:z_i=z_j}\log[z_i,\infty]
=\sum_{w\in\supp\cZ\setminus\{\infty\}}(\ord_w\cZ)^2\log[w,\infty].
\end{multline*}
Now the proof is complete.
\end{proof}

\begin{lemma}[global computation]\label{th:global}
 Let $k$ be a product formula field and $k_s$ the separable closure of $k$
 in $\overline{k}$.
Then for every adelic continuous weight $g=\{g_v:v\in M_k\}$
and every $k$-algebraic zeros divisor $\cZ$ on $\bP^1(k_s)$,
\begin{multline}
  \sum_{v\in M_k}N_v\left((\cZ,\cZ)_{g_v} 
+2\sum_{w\in\supp\cZ\setminus\{\infty\}}(\ord_w\cZ)^2\log[w,\infty]_v\right)\\
=-2(\deg\cZ)^2h_g(\cZ)
+2\sum_{v\in M_k}N_v\sum_{w\in\supp\cZ}(\ord_w\cZ)^2g_v(w).\label{eq:globalkey} 
\end{multline}
\end{lemma}

\begin{proof}
Let $P\in\bigcup_{d\in\bN}k[p_0,p_1]_d$ be a representative of $\cZ$.
Summing up $N_v\times$\eqref{eq:discrepancystrong} (for this $P$)
over all $v\in M_k$, we have
 \begin{multline*}
 \sum_{v\in M_k}N_v\biggl((\cZ,\cZ)_{g_v}
  +2\sum_{w\in\supp\cZ\setminus\{\infty\}}(\ord_w\cZ)^2\log[w,\infty]_v\\
  -2\cdot\sum_{w\in\supp\cZ}(\ord_w\cZ)^2g_v(w)\biggr)
 =-2(\deg\cZ)^2h_g(\cZ)
 \end{multline*}
by the product formula \eqref{eq:product}
(since $L(P(1,\cdot))\in k\setminus\{0\}$ and,
under the assumption that $\cZ$ is on $\bP^1(k_s)$, 
$D^*(\cZ|\overline{k})\in k\setminus\{0\}$)
and the definition \eqref{eq:definingheight} of $h_g(\cZ)$.
\end{proof}

\section{Estimates of regularized Fekete sums $(\cZ_\epsilon,\cZ_\epsilon)_g$}\label{sec:regularized}

\subsection{Local estimate}
Let $k$ be a field and $K$ an algebraic and metric augmentation of $k$.
Let $\cZ$ be a $k$-algebraic zeros divisor on $\bP^1(\overline{k})$,
which we regard as the Radon measure $\sum_{w\in\supp\cZ}(\ord_w\cZ)\delta_w$
on $\sP^1=\sP^1(K)$,
and 
$g$ be a continuous weight on $\sP^1$ 
such that $g$ is a $1/\kappa$-H\"older continuous function on $(\sP^1,\sd)$ 
for some $\kappa\ge 1$
having the $1/\kappa$-H\"older constant $C(g)\ge 0$.

\begin{lemma}\label{th:quantitativelocal}
For every $\epsilon>0$,
\begin{multline*}
 (\cZ_{\epsilon},\cZ_{\epsilon})_g\\
 \ge
(\cZ,\cZ)_g
+2\sum_{w\in\supp\cZ\setminus\{\infty\}}(\ord_w\cZ)^2\log[w,\infty]
-2\sum_{w\in\supp\cZ}(\ord_w\cZ)^2g(w)\\
+(C_{\operatorname{abs}}+\log\epsilon)\cdot(\cZ\times\cZ)(\diag_{\bP^1(\overline{k})})
-2(\deg\cZ)^2(\epsilon+C(g)\epsilon^{1/\kappa}).
\end{multline*}
\end{lemma}

\begin{proof}
Set $\eta(\epsilon)=C(g)\epsilon^{1/\kappa}$.
Then for every $\epsilon>0$, using \eqref{eq:regularization},
\begin{align*}
 &(\cZ_{\epsilon},\cZ_{\epsilon})_g-(\cZ,\cZ)_g
=\int_{\sP^1\times\sP^1}\Phi_g\rd(\cZ_{\epsilon}\times\cZ_{\epsilon})
-\int_{\sP^1\times\sP^1\setminus\diag_{\bP^1(K)}}\Phi_g\rd(\cZ\times\cZ)\\
=&
\sum_{w\in\supp\cZ}(\ord_w\cZ)^2
\int_{\sP^1\times\sP^1}\Phi_g\rd([w]_{\epsilon}\times[w]_{\epsilon})\\
&+\sum_{(z,w)\in\bP^1\times\bP^1\setminus\diag_{\bP^1}}
\left(\int_{\sP^1\times\sP^1}\Phi_g(\cS,\cS')
\rd([z]_{\epsilon}\times [w]_{\epsilon})(\cS,\cS')-\Phi_g(z,w)\right)\\
\ge&
\sum_{w\in\supp\cZ\setminus\{\infty\}}(\ord_w\cZ)^2
\left(
C_{\operatorname{abs}}+\log\epsilon-2\epsilon
+2\log[w,\infty]
-2\eta(\epsilon)
-2g(w))\right)\\
&+(\cZ(\{\infty\}))^2(C_{\operatorname{abs}}+\log\epsilon-2\epsilon-2\eta(\epsilon)-2g(\infty))+\\
&+\left((\deg\cZ)^2-(\cZ\times\cZ)(\diag_{\bP^1(\overline{k})})\right)(-2\epsilon-2\eta(\epsilon))\\
=&
\left((\cZ\times\cZ)(\diag_{\bP^1(\overline{k})})\right)(C_{\operatorname{abs}}+\log\epsilon-2\epsilon-2\eta(\epsilon))+\\
&
+2\sum_{w\in\supp\cZ\setminus\{\infty\}}(\ord_w\cZ)^2\log[w,\infty]
-2\sum_{w\in\supp\cZ}(\ord_w\cZ)^2g(w)+\\
&+\left((\deg\cZ)^2-(\cZ\times\cZ)(\diag_{\bP^1(\overline{k})})\right)(-2\epsilon-2\eta(\epsilon)),
\end{align*}
which completes the proof.
\end{proof}

\subsection{Global estimate}
Let $k$ be a product formula field.
Let $\cZ$ be a $k$-algebraic zeros divisor on $\bP^1(k_s)$, and
$g=\{g_v:v\in M_k\}$ be a placewise H\"older continuous
adelic normalized weight, so for every $v\in M_k$, 
$g_v$ is a normalized weight on $\sP^1(\bC_v)$ and
is a $1/\kappa_v$-H\"older continuous function
on $(\sP^1(\bC_v),\sd_v)$ for some $\kappa_v\ge 1$
having the $1/\kappa_v$-H\"older constant 
$C(g_v)\ge 0$.

\begin{lemma}\label{th:quantitativeFekete}
For every $v_0\in M_k$ and every $\epsilon>0$,
\begin{align*}
N_{v_0}(\cZ_{\epsilon},\cZ_{\epsilon})_{g_{v_0}}
\ge& 
 -2(\deg\cZ)^2h_g(\cZ)\\
&+(C_{\operatorname{abs}}+\log\epsilon)\cdot(\cZ\times\cZ)(\diag_{\bP^1(k_s)})
\cdot\sum_{v\in E_g\cup\{v_0\}}N_v\\
&-2(\deg\cZ)^2\cdot\sum_{v\in E_g\cup\{v_0\}}N_v
(\epsilon+C(g_v)\epsilon^{1/\kappa_{v_0}}).
\end{align*}
\end{lemma}  

\begin{proof}
Fix $v_0\in M_k$. Since $(\cZ_{\epsilon},\cZ_{\epsilon})_{g_v}\le 0$ 
for every $\epsilon>0$ and every $v\in M_k$ (see \S\ref{sec:negative}),
using also Lemma \ref{th:quantitativelocal}, we have
\begin{align*}
&N_{v_0}(\cZ_{\epsilon},\cZ_{\epsilon})_{g_{v_0}}
\ge 
\sum_{v\in E_g\cup\{v_0\}}N_v(\cZ_{\epsilon},\cZ_{\epsilon})_{g_{v_0}}\\
\ge
&\sum_{v\in E_g\cup\{v_0\}}N_v
\biggl((\cZ,\cZ)_{g_v}+\\
&+2\sum_{w\in\supp\cZ\setminus\{\infty\}}(\ord_w\cZ)^2\log[w,\infty]_v
-2\sum_{w\in\supp\cZ}(\ord_w\cZ)^2g_v(w)
\biggr)\\
&+(C_{\operatorname{abs}}+\log\epsilon)\cdot(\cZ\times\cZ)(\diag_{\bP^1(k_s)})
\cdot\sum_{v\in E_g\cup\{v_0\}}N_v\\
&-2(\deg\cZ)^2\cdot\sum_{v\in E_g\cup\{v_0\}}N_v
(\epsilon+C(g_v)\epsilon^{1/\kappa_{v_0}}).
\end{align*}
Moreover, since for every $v\in M_k\setminus E_g$, 
$g_v\equiv 0$ on $\sP^1(\bC_v)$ and $(\cZ,\cZ)_{g_v}\le 0$,
using also \eqref{eq:globalkey}, we have
\begin{align*}
&\sum_{v\in E_g\cup\{v_0\}}N_v\biggl((\cZ,\cZ)_{g_v}+\\
&+2\sum_{w\in\supp\cZ\setminus\{\infty\}}(\ord_w\cZ)^2\log[w,\infty]_v
-2\sum_{w\in\supp\cZ}(\ord_w\cZ)^2g_v(w)
\biggr)\\
\ge
&\sum_{v\in M_k}N_v\biggl((\cZ,\cZ)_{g_v}+\\
&+2\sum_{w\in\supp\cZ\setminus\{\infty\}}(\ord_w\cZ)^2\log[w,\infty]_v
-2\sum_{w\in\supp\cZ}(\ord_w\cZ)^2g_v(w)
\biggr)\\
=&
-2(\deg\cZ)^2h_g(\cZ),
\end{align*}
which completes the proof.
\end{proof}

\section{Proofs of Theorems \ref{th:adelicquantitative} and \ref{th:arith}}
\label{sec:quantitative}

\begin{proof}[Proof of Theorem $\ref{th:adelicquantitative}$]
Fix $v_0\in M_k$.
For every $v\in M_k$,
$g_v$ is a $1/\kappa_v$-H\"older continuous function
on $(\sP^1(\bC_v),\sd_v)$ for some $\kappa_v\ge 1$
having the $1/\kappa_v$-H\"older constant $C(g_v)\ge 0$.
Set $\epsilon=1/(\deg\cZ)^{2\kappa_{v_0}}$.
For every test function $\phi\in C^1(\sP^1(\bC_{v_0}))$,
 by \eqref{eq:schwarz} and Lemma \ref{th:quantitativeFekete},
 \begin{multline*}
 \left|\int_{\sP^1(\bC_{v_0})}\phi\rd\left(\frac{\cZ}{\deg\cZ}-\mu^g_{v_0}\right)\right|
 \le\frac{\Lip(\phi)_{v_0}}{(\deg\cZ)^{2\kappa_0}}
 +\frac{\langle\phi,\phi\rangle_{v_0}^{1/2}}{N_{v_0}^{1/2}}\times\\
 \times\Biggl(2\cdot h_g(\cZ)
 +(-C_{\operatorname{abs}}+2\kappa_{v_0}\log\deg\cZ)\cdot
 \frac{(\cZ\times\cZ)(\diag_{\bP^1(k_s)})}{(\deg\cZ)^2}
 \cdot\sum_{v\in E_g\cup\{v_0\}}N_v+\\ 
 +2\sum_{v\in E_g\cup\{v_0\}}N_v
 \left(\frac{1}{(\deg\cZ)^{2\kappa_0}}+\frac{C(g_v)}{(\deg\cZ)^2}\right) 
 \Biggr)^{1/2}, 
 \end{multline*}
 which
 completes the proof.
\end{proof}

\begin{proof}[Proof of Theorem $\ref{th:arith}$]
 Fix $v_0\in M_k$. For every $n\in\bN$, 
 $(\cZ_n,\cZ_n)_{g_v}\le 0$ if $v\in M_k\setminus E_g$.
 Hence by \eqref{eq:Feketeupper} (and $V_{g_v}=0$ for every $v\in M_k$)
 and \eqref{eq:globalkey},
 \begin{multline*}  
 N_{v_0}\frac{(\cZ_n,\cZ_n)_{g_{v_0}}}{(\deg\cZ_n)^2}
 +\#E_g\cdot o(1)\ge
 \sum_{v\in M_k}N_v\frac{(\cZ_n,\cZ_n)_{g_v}}{(\deg\cZ_n)^2}\\
 \ge-2\cdot h_g(\cZ_n)
 -2\frac{(\cZ_n\times\cZ_n)(\diag_{\bP^1(k_s)})}{(\deg\cZ_n)^2}\sum_{v\in E_g}N_v\sup_{\sP^1(\bC_v)}|g_v|\quad\text{as }n\to\infty,
 \end{multline*} 
 so that under the assumption that
 $(\cZ_n)$ has both small diagonals and small $g$-heights, we have
 $\liminf_{n\to\infty}(\cZ_n,\cZ_n)_{g_{v_0}}/(\deg\cZ_n)^2\ge 0=V_{g_{v_0}}$.
 Hence \eqref{eq:Feketeequiv} holds for $g_{v_0}$ and $(\cZ_n)$, and
 the proof is complete.
\end{proof} 

\section{Non-archimedean and complex dynamics}
\label{sec:dynamics}

\begin{fact}
 Let $k$ be a field.
 For a rational function $\phi\in k(z)$,
 we call $F_\phi=((F_{\phi})_0,(F_{\phi})_1)\in\bigcup_{d\in\bN\cup\{0\}}(k[p_0,p_1]_d\times k[p_0,p_1]_d)$
 is a {\itshape lift} of $\phi$ if
 $\pi\circ F_{\phi}=\phi\circ\pi$ on $k^2\setminus\{0\}$
 and, in addition, $F_{\phi}^{-1}(0)=\{0\}$ when $\deg\phi>0$.
 The latter non-degeneracy condition is equivalent to 
 $\Res(F_{\phi}):=\Res((F_{\phi})_0,(F_{\phi})_1)\neq 0$;
 for the definition of the homogeneous resultant $\Res(P,Q)\in k$ between
 $P,Q\in\bigcup_{d\in\bN\cup\{0\}}k[p_0,p_1]_d$, 
 see e.g. \cite[\S2.4]{SilvermanDynamics}.
 Such a lift $F_{\phi}$ of $\phi$ is unique up to multiplication 
 in $k^*$, and is in $k[p_0,p_1]_{\deg\phi}\times k[p_0,p_1]_{\deg\phi}$.
\end{fact}

Let $K$ be an algebraically closed field that is complete with respect to a
non-trivial absolute value $|\cdot|$.

\subsection{Dynamical Green function $g_f$ on $\sP^1$}

For the foundation of a potential theoretical study of
dynamics on the Berkovich projective line, see Baker--Rumely
\cite{BR10} and Favre--Rivera-Letelier \cite{FR09} for non-archimedean $K$
and, e.g., \cite[\S VIII]{BM01} for archimedean $K(\cong\bC)$.

\begin{fact}\label{eq:action}
 Let $\phi\in K(z)$ be a rational function of degree $d_0\in\bN\cup\{0\}$.
 The action of $\phi$ on $\bP^1=\bP^1(K)$
 uniquely extends to a continuous endomorphism on $\sP^1=\sP^1(K)$. When
 $d_0>0$, the extended $\phi$
 is surjective, open, and discrete and preserves
 $\bP^1$ and $\sH^1=\sH^1(K)$, 
 the local degree function $z\mapsto\deg_z\phi$ on $\bP^1$ also canonically
 extends to $\sP^1$,
 and the (mapping) degree of the extended $\phi:\sP^1\to\sP^1$ still equals $d_0$
 (cf.\ \cite[\S 2.3, \S 9]{BR10}, \cite[\S 6.3]{Benedetto10}):
 in particular, the extended action $\phi$ on $\sP^1$
 induces a push-forward $\phi_*$ and a pullback $\phi^*$ 
 on the spaces of continuous functions and of
 Radon measures on $\sP^1$. When $d_0=0$,
 the extended $\phi$ is still constant, and we set $\phi^*\mu:=0$ on $\sP^1$
 for every Radon measure $\mu$ on $\sP^1$ by convention.
 Let $F_{\phi}\in K[p_0,p_1]_{\deg\phi}\times K[p_0,p_1]_{\deg\phi}$
 be a lift of $\phi$. The function 
 \begin{gather}
  T_{F_{\phi}}:=\log\|F_{\phi}(\cdot/\|\cdot\|)\|=\log\|F_\phi\|-(\deg\phi)\log\|\cdot\|\label{eq:descend}
 \end{gather} 
 on $K^2\setminus\{0\}$ descends to $\bP^1$ and in turn extends 
 continuously to $\sP^1$, satisfying 
 $\Delta T_{F_{\phi}}=\phi^*\Omega_{\can}-(\deg\phi)\Omega_{\can}$
 on $\sP^1$ (see, e.g., \cite[Definition 2.8]{OkuCharacterization}).
 Moreover,
 $\phi$ is a Lipschitz continuous
 endomorphism on $(\sP^1,\sd)$ and 
 $T_{F_{\phi}}$ is a Lipschitz continuous function on $(\sP^1,\sd)$
 (for non-archimedean $K$, see \cite[Proposition 9.37]{BR10}).
 For every $n\in\bN$, 
 $F_{\phi}^n\in K[p_0,p_1]_{\deg\phi^n}\times K[p_0,p_1]_{\deg\phi^n}$
 is a lift of $\phi^n$. 
\end{fact}

Let $f\in K(z)$ be a rational function of degree $d>1$ and
$F\in K[p_0,p_1]_d\times K[p_0,p_1]_d$ a lift of $f$.
Then the uniform limit 
$g_F:=\lim_{n\to\infty}T_{F^n}/d^n$
on $\sP^1$ exists, and more precisely, for every $n\in\bN$,
\begin{gather}
 \sup_{\sP^1}\left|g_F-\frac{T_{F^n}}{d^n}\right|
 \le
 \frac{\sup_{\sP^1}|T_F|}{d^n(d-1)}.\label{eq:uniformconvGreen}
 \end{gather}
The limit $g_F$ is called
 the {\itshape dynamical Green function of $F$} on $\sP^1$ and
is a continuous weight on $\sP^1$.
The probability Radon measure
\begin{gather*}
 \mu_f:=\mu^{g_F}=\Delta g_F+\Omega_{\can}
 =\lim_{n\to\infty}\frac{(f^n)^*\Omega_{\can}}{d^n}\quad\text{weakly on }\sP^1
\end{gather*}
is independent of choices of $F$ and satisfies 
$f^*\mu_f=d\cdot\mu_f$ on $\sP^1$, and 
is called the {\itshape $f$-equilibrium $($or canonical$)$ measure} on $\sP^1$.
Moreover, $g_F$ is a H\"older continuous function on $(\sP^1,\sd)$
(for non-archimedean $K$, see \cite[\S 6.6]{FR06}).
The remarkable {\itshape energy formula}
\begin{gather}
 V_{g_F}=-\frac{\log|\Res F|}{d(d-1)}\label{eq:energyformula}
\end{gather}
was first established by DeMarco \cite{DeMarco03} for archimedean $K$
and was generalized to rational functions defined over a number field 
by Baker--Rumely \cite{BR06} (for a simple proof of \eqref{eq:energyformula}
which also works for general $K$, see Baker--Rumely \cite[Appendix A]{Baker09}
 or Stawiska and the author \cite[Appendix]{OS11}). 
The {\itshape dynamical Green function $g_f$ of $f$ on} $\sP^1$ is 
the unique normalized weight on $\sP^1$ such that $\mu^{g_f}=\mu_f$, i.e.,
for any lift $F$ of $f$,
$g_f\equiv g_F+V_{g_F}/2$ on $\sP^1$.

\subsection{Berkovich space version of the quasiperiodicity region $\cE_f$}
For non-archimedean dynamics,
see \cite[\S 10]{BR10}, \cite[\S2.3]{FR09}, \cite[\S 6.4]{Benedetto10}.
For complex dynamics, see, e.g., \cite{Milnor3rd}.

Let $f\in K(z)$ be a rational function of degree $>1$.
The {\itshape Berkovich} Julia set of $f$ is
$\sJ(f):=\{z\in\sP^1:\bigcap_{U:\text{ open neighborhood of }z\text{ in }\sP^1}
\left(\bigcup_{n\in\bN}f^n(U)\right)=\sP^1\setminus E(f)\}$,
where $E(f):=\{a\in\bP^1:\#\bigcup_{n\in\bN}f^{-n}(a)<\infty\}$
is the {\itshape exceptional set} of $f$,
and the {\itshape Berkovich} Fatou set is $\sF(f):=\sP^1\setminus\sJ(f)$.
By definition, $\sJ(f)$ is closed and $\sF(f)$ is open in $\sP^1$,
both $\sJ(f)$ and $\sF(f)$ are totally invariant under $f$, and
$\sJ(f)$ has no interior point unless $\sJ(f)=\sP^1$. Moreover,
the {\itshape classical} Julia set $\sJ(f)\cap\bP^1$ (resp.\ the {\itshape classical}
Fatou set $\sF(f)\cap\bP^1$) coincides with the set of
all non-equicontinuity points (resp.\ the region of equicontinuity)
of the family $\{f^n:n\in\bN\}$ as a family of endomorphisms on $(\bP^1,[z,w])$. 


A component $U$ 
of $\sF(f)$ is called a {\itshape Berkovich} Fatou component of $f$,
and is said to be {\itshape cyclic} under $f$ if $f^n(U)=U$ for some $n\in\bN$, 
which is called a {\itshape period} of $U$ under $f$.
A cyclic Berkovich Fatou component $U$ of $f$ 
having a period $n\in\bN$ is called a {\itshape singular domain}
of $f$ if $f^n:U\to U$ is injective (following Fatou \cite[Sec. 28]{Fatou1920trois}). Let $\cE_f$ 
be the set of all points $\cS\in\sP^1$ having an open neighborhood $V$
in $\sP^1$ such that $\liminf_{n\to\infty}\sup_{V\cap\bP^1}[f^n,\Id]=0$,
which is a Berkovich space version of 
Rivera-Letelier's {\itshape quasiperiodicity region} of $f$:
when $K$ is archimedean,
$\cE_f$ {\itshape coincides with} the union of all singular domains of $f$, 
and when $K$ is non-archimedean, $\cE_f$ is 
still open and forward invariant under $f$ and 
{\itshape is contained in} the union of all singular domains of $f$
 (cf.\ \cite[Lemma 4.4]{OkuCharacterization}).

The following function $T_*$ is Rivera-Letelier's {\itshape iterative logarithm} of $f$ 
on $\cE_f\cap\bP^1$, which is a non-archimedean counterpart of the uniformization 
of a Siegel disk or an Herman ring of $f$. 

\begin{theorem}[{\cite[\S 3.2, \S4.2]{Juan03}. See also \cite[Th\'eor\`eme 2.15]{FR09}}]\label{th:uniformization}
 Suppose that $K$ is non-archimedean and
 has characteristic $0$ and residual characteristic $p$. 
 Let $f\in K(z)$ be a rational function on $\bP^1$
 of degree $>1$ and suppose that $\cE_f\neq\emptyset$, which implies 
 $p>0$ by \cite[Lemme 2.14]{FR09}.
 Then for every component $Y$ of $\cE_f$ not containing $\infty$,
 there are $k_0\in\bN$, a continuous action 
 $T:\bZ_p\times (Y\cap K)\ni (\omega,y)\mapsto T^\omega(y)\in Y\cap K$,
 and a non-constant $K$-valued holomorphic function $T_*$ on $Y\cap K$
 such that for every $m\in\bZ$, $(f^{k_0})^m=T^m$ on $Y\cap K$, 
 that for each $\omega\in\bZ_p$, $T^{\omega}$ is a biholomorphism on $Y\cap K$,
 and that
 for every $\omega_0\in\bZ_p$,
 \begin{gather}
  \lim_{\bZ_p\ni\omega\to\omega_0}\frac{T^{\omega}-T^{\omega_0}}{\omega-\omega_0}
=T_*\circ T^{\omega_0}\quad\text{locally uniformly on }Y\cap K.\label{eq:logarithm}
 \end{gather} 
\end{theorem}

\subsection{Fundamental relationship between $\mu_f$ and $\sJ(f)$}

The inclusion $\supp\mu_f\subset\sJ(f)$ is classical when $K$ is archimedean,
but is no trivial from the definition of $\sJ(f)$
when $K$ is non-archimedean; for an elementary proof of it,
see \cite[Proof of Theorem 2.18]{OkuCharacterization}. The equality $\supp\mu_f=\sJ(f)$ indeed holds,
but we will dispense with the reverse (and easier)
inclusion $\sJ(f)\subset\supp\mu_f$.

\section{Proofs of Theorems \ref{th:quantitativedynamics} and \ref{th:qualitativedynamics}}\label{subsec:adelicGreen}
Let $k$ be a product formula field. 
The proof of the following is based not only on \eqref{eq:product}
but also on the elimination theory (and the strong triangle inequality).

\begin{theorem}[{\cite[Lemma 3.1]{BR06}}]\label{th:elimination}
 Let $k$ be a product formula field.
 For every $\phi\in k(z)$ and
 every lift $F_{\phi}\in k[p_0,p_1]_{\deg\phi}\times k[p_0,p_1]_{\deg\phi}$ 
 of $\phi$, 
there exists a finite subset $E_{F_{\phi}}$ in $M_k$ 
 containing all the infinite places of $k$ such that
 for every place $v\in M_k\setminus E_{F_{\phi}}$, $|\Res F_{\phi}|_v=1$ and
 $\|F_{\phi}(\cdot)\|_v=\|\cdot\|_v^{\deg\phi}$ on $\bC_v$.
\end{theorem}

Let $f\in k(z)$ be a rational function of degree $>1$
and $F\in k[p_0,p_1]_d\times k[p_0,p_1]_d$ a lift of $f$.
Then the family $\hat{g}_f=\{g_{f,v}:v\in M_k\}$ is
an adelic normalized weight, where for every $v\in M_k$,
$g_{f,v}$ is the dynamical Green function of $f$ on $\sP^1(\bC_v)$;
for, letting $g_{F,v}$ be the dynamical Green function of $F$
on $\sP^1(\bC_v)$ for each $v\in M_k$ and
$E_F$ be a finite subset in $M_k$
obtained by Theorem $\ref{th:elimination}$ applied to $F$,
for every $v\in M_k\setminus E_F$, we have
$T_{F^n,v}\equiv 0$ on $\sP^1(\bC_v)$ for every $n\in\bN$ so
$g_{f,v}\equiv g_{F,v}\equiv 0$ on $\sP^1(\bC_v)$.
We call the adelic normalized weight $\hat{g}_f=\{g_{f,v}:v\in M_k\}$
and the adelic probability measure
$\hat{\mu}_f
:=\mu^{\hat{g}_f}$
the {\itshape adelic} dynamical Green function of $f$
and the {\itshape adelic $f$-equilibrium $($or canonical$)$ measure},
respectively.
Here, for every $v\in M_k$, 
$\mu_{f,v}:=\mu^{g_{f,v}}=\mu^{\hat{g}_f}_v$ (as in Section \ref{sec:intro})
is the $f$-equilibrium (or canonical) measure on $\sP^1(\bC_v)$.

Once the following is at our disposal, 
Theorems \ref{th:adelicquantitative} and \ref{th:arith} will
yield Theorems \ref{th:quantitativedynamics} and \ref{th:qualitativedynamics},
respectively.

\begin{lemma}\label{th:heightsdynamics}
Let $k$ be a product formula field.
Let $f,a\in k(z)$ be rational functions and suppose that $d:=\deg f>1$. Then
the sequence $([f^n=a])$ of $k$-algebraic zeros divisors on
$\bP^1(\overline{k})$
has strictly small $\hat{g}_f$-heights in that
$\limsup_{n\to\infty}(d^n+\deg a)\cdot h_{\hat{g}_f}([f^n=a])<\infty$.
\end{lemma}

\begin{proof} 
 Let $F\in k[p_0,p_1]_d\times k[p_0,p_1]_d$ and
 $A\in k[p_0,p_1]_{\deg a}\times k[p_0,p_1]_{\deg a}$ 
 be lifts of $f,a$, respectively (then for every $n\in\bN$,
 $F^n\wedge A\in k[p_0,p_1]_{d^n+\deg a}\times k[p_0,p_1]_{d^n+\deg a}$
 is a representative of $[f^n=a]$ if $f^n\not\equiv a$).  
 Let $E_F,E_A$ be finite subsets in $M_k$ obtained by
 applying Theorem \ref{th:elimination} to $F,A$, respectively,
 so that for every $v\in M_k\setminus(E_F\cup E_A)$ and every $n\in\bN$,
 $T_{F^n,v}\equiv T_{A,v}\equiv 0$ and $g_{F,v}\equiv 0$ on $\sP^1(\bC_v)$.
 For every $v\in M_k$ and every $n\in\bN$ large enough, 
 since $|F^n\wedge A|_v\le\|F^n\|_v\|A\|_v$ on $\bC_v^2\setminus\{0\}$, 
 we have $\log S_{F^n\wedge A,v}\le T_{F^n,v}+T_{A,v}$ on $\bP^1(\bC_v)$ and
 in turn on $\sP^1(\bC_v)$ (recall that $S_{F^n\wedge A,v}
 =|(F^n\wedge A)(\cdot/\|\cdot\|_v)|_v$ on $\bP^1(\bC_v)$), so that
 using also $g_{f,v}\equiv g_{F,v}+V_{g_{F,v}}/2$ on $\sP^1(\bC_v)$,
 \begin{gather*}
  \frac{\log S_{F^n\wedge A,v}}{d^n+\deg a}-g_{f,v}
  \le\frac{T_{F^n,v}+T_{A,v}}{d^n+\deg a}-
  \left(g_{F,v}+\frac{1}{2}V_{g_{F,v}}\right)\quad\text{on }\sP^1(\bC_v).
 \end{gather*}
 Hence by the definition \eqref{eq:definingheight} of $h_{\hat{g}_f}$,
 the Jensen-type formula \eqref{eq:formulaheights},
 the energy formula \eqref{eq:energyformula} (and $\Res F\in k\setminus\{0\}$),
 and \eqref{eq:product},
we have
 \begin{align*}
 &h_{\hat{g}_f}([f^n=a])\\
\le&
\sum_{v\in M_k}N_v\int_{\sP^1(\bC_v)}\left(\frac{T_{F^n,v}+T_{A,v}}{d^n+\deg a}-g_{F,v}\right)\rd\mu_{f,v}-\frac{3}{2}\sum_{v\in M_k}N_v\cdot V_{g_{F,v}}\\
=&\sum_{v\in E_F\cup E_A}N_v\int_{\sP^1(\bC_v)}\left(\frac{T_{F^n,v}+T_{A,v}}{d^n+\deg a}-g_{F,v}\right)\rd\mu_{f,v}
=O(d^{-n})\quad\text{as }n\to\infty,
 \end{align*}
where the final order estimate is by
\eqref{eq:uniformconvGreen} and $\#(E_F\cup E_A)<\infty$.
\end{proof}

We omit the proof of the following characterization of $h_{\hat{g}_f}$,
which we will dispense with in this article.

\begin{lemma}
 Let $k$ be a product formula field. 
 Then for every rational function $f\in k(z)$ of degree $d>1$, 
 the $\hat{g}_f$-height function $h_{\hat{g}_f}$ coincides with
 the {\rm Call--Silverman $f$-dynamical 
 $($or canonical$)$ height function} in that
 for every $k$-algebraic zeros divisor $\cZ$ on $\bP^1(\overline{k})$,
 $($$f_*\cZ$ is also
 a $k$-algebraic zeros divisor on $\bP^1(\overline{k})$ and$)$
 the equality $(h_{\hat{g}_f}\circ f_*)(\cZ)=(d\cdot h_{\hat{g}_f})(\cZ)$ holds.
\end{lemma}

\section{Proofs of Theorems \ref{th:nonlinearity} and \ref{th:identity}}
\label{eq:rotation}

Let $K$ be an algebraically closed field that is complete with respect to a
non-trivial absolute value $|\cdot|$. For subsets $A,B\subset\bP^1$, set
$[A,B]:=\inf_{z\in A,z'\in B}[z,z']$.

Let $f\in K(z)$ be a rational function of degree $d>1$
and $a\in K(z)$ a rational function. Let $N\in\bN$ be so large that
$f^n\not\equiv a$ if $n>N$. Then
$(\overline{\bigcup_{n\in\bN:n>N}\supp[f^n=a]}\cup\sJ(f))\cap\bP^1$
is closed in $\bP^1$. 

\begin{lemma}\label{th:limit}
 Suppose that $K$ has characteristic $0$.
 For every chordal disk $D$ in $\bP^1$ of radius $>0$
 satisfying that $\liminf_{n\to\infty}\sup_D[f^n,a]=0$,
 we have $(i)$ $a(D)\subset\cE_f$ and
 $(ii)$
 $D\setminus(\overline{\bigcup_{n\in\bN:n>N}\supp[f^n=a]}\cup\sJ(f))\neq\emptyset$,
and moreover, $(iii)$
 there is a chordal disk $D'$ in $\bP^1\setminus\sJ(f)$ of radius $>0$
 such that $\liminf_{n\to\infty}[f^n(D'),a(D')]>0$.
\end{lemma}

\begin{proof}[Proof of $(i)$]
By $\liminf_{n\to\infty}\sup_D[f^n,a]=0$, 
there is
a sequence $(n_j)$ in $\bN$ tending to $\infty$ as $j\to\infty$
 such that $\lim_{j\to\infty}\sup_D[f^{n_j},a]=0$ and that
 $\lim_{j\to\infty}(n_{j+1}-n_j)=\infty$.
 For every $z\in D$,
 set $D'':=\{w\in\bP^1:[w,a(z)]\le r\}$
 in $a(D)$ for $r>0$ small enough. Then
 $\liminf_{j\to\infty}\sup_{D''}[f^{n_{j+1}-n_j},\Id]
 \le\limsup_{j\to\infty}\sup_D[f^{n_{j+1}},f^{n_j}]=0$, so that $a(z)\in\cE_f$.
 Hence $a(D)\subset\cE_f$.
 \end{proof}

\begin{proof}[Proof of $(ii)$]
When $K$ is archimedean, let $Y$ be the component of $\cE_f$
containing $a(D)$, which is by the first assertion
either a Siegel disk or an Herman ring of $f$. 
Setting $k_0:=\min\{n\in\bN:f^n(Y)=Y\}\in\bN$, 
there are a sequence $(n_j)$ in $\bN$ and an $N\in\bN$
such that $f^{n_N}(D)\subset Y$,
that $k_0|(n_j-n_N)$ for every $j\ge N$, and 
that $a=\lim_{j\to\infty}(f^{k_0})^{(n_j-n_N)/k_0}\circ f^{n_N}$
uniformly on $D$. Then $D\cap\sJ(f)=\emptyset$.
Let $\lambda\in\bC$ be the rotation number of $Y$ in that
there exists
a holomorphic injection $h:Y\to\bC$ such that $h\circ f^{k_0}=\lambda\cdot h$ on
$Y$. Then $|\lambda|=1$ but $\lambda$ is not a root of the unity (by $d>1$). Choosing
a subsequence of $(n_j)$ if necessary, 
$\lambda_a:=\lim_{j\to\infty}\lambda^{(n_j-n_N)/k_0}\in\bC$ exists. 
For every $n\ge n_N$, if $n-n_N$ is not divided by $k_0$, then
$D\cap\supp[f^n=a]=\emptyset$. For every $n\ge n_N$,
if $k_0|(n-n_N)$, then
$h\circ f^n-h\circ a=(\lambda^{(n-n_N)/k_0}-\lambda_a)\cdot(h\circ f^{n_N})$
on $D$, so $(D\setminus(h\circ f^{n_N})^{-1}(0))\cap\supp[f^n=a]=\emptyset$
if $n$ is large enough.

When $K$ is non-archimedean, let $Y$ be the component of $\cE_f$ containing
$a(D)$. Without loss of generality, we assume that $\infty\not\in Y$, and then
applying Theorem \ref{th:uniformization} to this $Y$, we obtain
$p\in\bN$, $k_0\in\bN$,
$T$, and $T_*$ as in Theorem \ref{th:uniformization}.
There are a sequence $(n_j)$ in $\bN$ and an $N\in\bN$
such that $f^{n_N}(D)\subset Y$, that $k_0|(n_j-n_N)$ for every $j\ge N$,
and that $a=\lim_{j\to\infty}(f^{k_0})^{(n_j-n_N)/k_0}\circ f^{n_N}$
uniformly on $D$. Then $D\cap\sJ(f)=\emptyset$.
Choosing a subsequence of $(n_j)$ if necessary, 
$\omega_a:=\lim_{j\to\infty}(n_j-n_N)/k_0\in\bZ_p$ exists. 
For every $n\ge n_N$, if $n-n_N$ is not divided by $k_0$,
then $D\cap\supp[f^n=a]=\emptyset$. 
For every $n\ge n_N$, if $k_0|(n-n_N)$, then
\begin{gather}
 f^n-a=(T^{(n-n_N)/k_0}-T^{\omega_a})\circ f^{n_N}\label{eq:unifnonarch}
\end{gather}
on $D$. Choose a ($K$-closed) disk $B=\{z\in K:|z-b|\le r\}$ in $D$
for some $b\in D\setminus\{\infty\}$ and $r\in|K^*|$
small enough, 
and fix $\epsilon\in|K^*|$ so small that
setting $Z_{\epsilon}:=\bigcup_{w\in B\cap(((T_*\circ T^{\omega_a})\circ f^{n_N})^{-1}(0))}\{z\in B:|z-w|<\epsilon\}$, $B\setminus Z_\epsilon\neq\emptyset$.
By the maximum modulus principle from the rigid analysis
(cf.\ \cite[\S6.2.1, \S7.3.4]{BGR}), 
$\min_{z\in f^{n_N}(B\setminus Z_{\epsilon})}|T_*\circ T^{\omega_a}(z)|>0$, so 
that by the uniform convergence \eqref{eq:logarithm} and 
the equality \eqref{eq:unifnonarch},
$(B\setminus Z_{\epsilon})\cap\supp[f^n=a]=\emptyset$ if $n$ is large enough. 
\end{proof}

\begin{proof}[Proof of $(iii)$ in Lemma $\ref{th:limit}$]
By the first assertion,
there is a unique singular domain $U$ of $f$ containing $a(D)$.
Fix $n_0\in\bN$ such that $f^{n_0}(U)=U$, and
set $\mathcal{C}:=\bigcup_{j=0}^{n_0-1}f^j(U)$.
Then there is a component $V$ of $f^{-1}(\mathcal{C})\setminus\mathcal{C}$
since $f:\mathcal{C}\to\mathcal{C}$ is injective and $d>1$.
Fix a chordal disk $D''$ of radius $>0$
in $a^{-1}(V)\cap(\bP^1\setminus\sJ(f))$, so that $a(D'')\subset V\subset f^{-1}(\mathcal{C})\setminus\mathcal{C}$.
If $a(D'')\cap\bigcup_{n\in\bN\cup\{0\}}f^n(D'')=\emptyset$, then
we are done by setting $D'=\{z\in\bP^1:[z,b]\le r\}$ for some $b\in D''$
and $r>0$ small enough.
If there is $N\in\bN\cup\{0\}$ such that $a(D'')\cap f^N(D'')\neq\emptyset$,
then setting $D':=\{z\in\bP^1:[z,b]\le r\}$ for some
$b\in D''\cap f^{-N}(a(D''))$ and $r>0$ small enough,
we have $\liminf_{n\to\infty}[a(D'),f^n(D')]>0$
since $a(D')\cap\bigcup_{n\ge N+1}f^n(D')\subset a(D'')\cap\bigcup_{n\in\bN}f^n(a(D''))\subset V\cap\mathcal{C}=\emptyset$.
\end{proof}

\begin{lemma}\label{th:modify}
 For every 
 $w_0\in\bP^1\setminus(\overline{\bigcup_{n\in\bN:n>N}\supp[f^n=a]}\cup\sJ(f))$,
 there is a function $\phi_0\in C^1(\sP^1)$ such that
 $\phi_0\equiv\log[w_0,\cdot]_{\can}$ on
 $\bigcup_{n\in\bN:n>N}\supp[f^n=a]\cup\sJ(f)$.
\end{lemma}
\begin{proof} 
 Fix $w_0\in\bP^1\setminus(\overline{\bigcup_{n\in\bN:n>N}\supp[f^n=a]}\cup\sJ(f))$. 
 Without loss of generality, we can assume that $w_0\neq\infty$, and
 fix $\epsilon>0$ so small that
 $\{\cS\in\sP^1:|\cS-w_0|_{\infty}\le\epsilon\}
 \subset\sP^1\setminus
 (\overline{\bigcup_{n\in\bN}\supp[f^n=a]}\cup\sJ(f))$
 (recall \S \ref{sec:Hsia}, \S \ref{sec:tree} here).
 When $K$ is non-archimedean, 
by the definition of $\pi_{\epsilon}:\sA^1\to\sA^1$,
we have $\{\cS\in\sP^1:\cS\preceq\pi_{\epsilon}(w_0)\}
=\{\cS\in\sP^1:|\cS-w_0|_{\infty}\le\epsilon\}$, and
the function
 \begin{gather*}
 \cS\mapsto\phi_0(\cS):=
\begin{cases}
\log[w_0,\pi_{\epsilon}(w_0)]_{\can} & \text{if }\cS\preceq\pi_{\epsilon}(w_0)\\
 \log[w_0,\cS]_{\can} & \text{otherwise} 
\end{cases}\quad\text{on }\sP^1
\end{gather*}
is in $C^1(\sP^1)$ since it is is continuous on $\sP^1$,
locally constant on $\sP^1$ except for the segment $\mathcal{I}$
in $\sH^1$ joining $\pi_{\epsilon}(w_0)$ and $\cS_{\can}$, and
is linear on $\mathcal{I}$ 
with respect to the length parameter 
induced by the hyperbolic metric $\rho$ on $\sH^1$.
When $K$ is archimedean (so $\sP^1\cong\bP^1$),
there is a function $\phi_0\in C^1(\bP^1)$ satisfying
\begin{gather*}
 z\mapsto\phi_0(z)=
\begin{cases}
 \int_{\bP^1}\log[w_0,w]\rd[z]_{\epsilon/2}(w) & \text{if }|z-w_0|\le\epsilon/2,\\
 \log[w_0,z] & \text{if }|z-w_0|\ge\epsilon\text{ or }z=\infty.
\end{cases}
\end{gather*}
 In both cases, the $\phi_0\in C^1(\sP^1)$ satisfies the desired property.
\end{proof}

\begin{fact}
 For rational functions $\phi,\psi\in K(z)$,
 the {\itshape chordal proximity function} 
 \begin{gather*}
 \cS\mapsto[\phi,\psi]_{\can}(\cS)\quad\text{on }\sP^1
 \end{gather*}
 between $\phi$ and $\psi$ is the unique continuous extension of
 the function $z\mapsto[\phi(z),\psi(z)]$ on $\bP^1$
 to $\sP^1$ (see \cite[Proposition 2.9]{OkuCharacterization}
 for its construction, and also
 \cite[Remark 2.10]{OkuCharacterization}), and
 for every continuous weight $g$ on $\sP^1$, we also define
 its weighted version
 by $\Phi(\phi,\psi)_g:=\log[\phi,\psi]_{\can}-g\circ\phi-g\circ\psi$
 on $\sP^1$.
\end{fact}

For every $n\in\bN$ such that $f^n\not\equiv a$, recall
the following {\itshape Riesz decomposition}
\begin{gather}
\Phi(f^n,a)_{g_f}=U_{g_f,[f^n=a]-(d^n+\deg a)\mu_f}
-U_{g_f,a^*\mu_f}
+\int_{\sP^1}\Phi(f^n,a)_{g_f}\rd\mu_f\label{eq:rewrittenRiesz}
\end{gather}
on $\sP^1$, and also
$U_{g_f,a^*\mu_f}
=g_f\circ a+U_{g_f,a^*\Omega_{\can}}-\int_{\sP^1}(g_f\circ a)\rd\mu_f$ on $\sP^1$
(\cite[Lemma 2.19]{OkuCharacterization}).

\begin{proof}[Proof of Theorem $\ref{th:nonlinearity}$]
Let $k$ be a product formula field of characteristic $0$.
Let $f\in k(z)$ be a rational function of degree $d>1$ and
$a\in k(z)$ a rational function of degree $>0$.
Fix $v\in M_k$.
Let $D$ be a chordal disk in $\bP^1(\bC_v)$
of radius $>0$, and assume that $\liminf_{n\to\infty}\sup_D[f^n,a]_v=0$; otherwise
we are done. By Lemma \ref{th:limit}, 
there are not only a point
$w_0\in D\setminus(\overline{\bigcup_{n\in\bN:n>N}[f^n=a]}\cup\sJ(f)_v)$
but also a chordal disk $D'$ in $\bP^1(\bC_v)\setminus\sJ(f)_v$ of radius $>0$ 
such that $\liminf_{n\to\infty}[f^n(D'),a(D')]_v>0$.
Fix a point $w_1\in D'$. Then also
$w_1\in\bP^1\setminus(\overline{\bigcup_{n\in\bN:n>N}[f^n=a]}\cup\sJ(f)_v)$. 

For every $n\in\bN$ large enough
and each $j\in\{0,1\}$, by
 \eqref{eq:rewrittenRiesz},
\begin{multline}
 \log[f^n(w_j),a(w_j)]_v-g_{f,v}(f^n(w_j))-g_{f,v}(a(w_j))\\
 =U_{g_{f,v},[f^n=a]-(d^n+\deg a)\mu_{f,v}}(w_j)
 -U_{g_{f,v},a^*\mu_{f,v}}(w_j)
+\int_{\sP^1(\bC_v)}\Phi(f^n,a)_{g_{f,v}}\rd\mu_{f,v},\label{eq:twopoints}  
\end{multline}
so that
taking the difference of both sides in \eqref{eq:twopoints}
for each $j\in\{0,1\}$,
since $g_{f,v}$ and $U_{g_{f,v},a^*\mu_{f,v}}$
are bounded on $\sP^1(\bC_v)$, we have
\begin{align*}
& \log[f^n(w_0),a(w_0)]_v-\log[f^n(w_1),a(w_1)]_v\\
=&\int_{\sP^1(\bC_v)}\log[w_0,\cS']_{\can,v}\rd([f^n=a]-(d^n+\deg a)\mu_f)(\cS')\\
 &-\int_{\sP^1(\bC_v)}\log[w_1,\cS']_{\can,v}\rd([f^n=a]-(d^n+\deg a)\mu_f)(\cS')
 +O(1)
\end{align*}
as $n\to\infty$.
In the left hand side, by the choice of $w_0$ and $w_1$, we have
$\log\sup_D[f^n,a]_v\ge\log[f^n(w_0),a(w_0)]_v$ and
 $\liminf_{n\to\infty}\log[f^n(w_1),a(w_1)]_v
 \ge\liminf_{n\to\infty}\log[f^n(D'),a(D')]_v>-\infty$, so that
as $n\to\infty$,
\begin{gather*}
 \log\sup_D[f^n,a]_v +O(1)
 \ge\log[f^n(w_0),a(w_0)]_v-\log[f^n(w_1),a(w_1)]_v.
\end{gather*}
In the right hand side, for each $j\in\{0,1\}$, by
Lemma \ref{th:modify} applied to $w_j$, the inclusion 
$\supp\mu_f\subset\sJ(f)$,
and Theorem \ref{th:quantitativedynamics}
(and $k_s=\overline{k}$
under the assumption that $k$ has characteristic $0$), we have
\begin{multline*}
 \int_{\sP^1(\bC_v)}\log[w_j,\cS']_{\can,v}\rd([f^n=a]-(d^n+\deg a)\mu_f)(\cS')\\
=O\left(\sqrt{n\cdot([f^n=a]\times[f^n=a])(\diag_{\bP^1(\overline{k})})}\right)
\quad\text{as }n\to\infty.
\end{multline*}
The final three estimates complete the proof of \eqref{eq:nonlin} for this $v\in M_k$.
\end{proof}

\begin{fact}
 For a rational function $f(z)\in k(z)$ over a field $k$, a point
 $w\in\bP^1(\overline{k})$ is called a {\itshape multiple} periodic point
 of $f$ if $w\in\supp[f^n=\Id]$ and $[f^n=\Id](\{w\})>1$ for some $n\in\bN$.
 For a rational function $f(z)\in k(z)$ over a field $k$
 of characteristic $0$, there are {\itshape at most finitely many} multiple periodic points
 of $f$ in $\bP^1(\overline{k})$; this is well-known in the case
 that $k=\bC$
 (see, e.g., \cite[\S13]{Milnor3rd}), and holds in general
 {\itshape by the Lefschetz principle} (see, e.g., \cite{Eklof73}).
\end{fact} 

\begin{proof}[Proof of Theorem $\ref{th:identity}$]
As seen in the above,
$f$ has at most finitely many
multiple periodic points in $\bP^1(\overline{k})$, and
for every multiple periodic point $w$ of $f$,
setting $p=p_w:=\min\{n\in\bN:[f^n=\Id](\{w\})>1\}$,
by the (formal) power series expansion
$f^p(z)=w+(z-w)+C(z-w)^{[f^p=\Id](\{w\})}+\cdots$ of $f^p$ around $w$,
we also have $\sup_{n\in\bN}[f^n=\Id](\{w\})\le[f^p=\Id](\{w\})$
under the characteristic $0$ assumption.
Hence
$\sup_{n\in\bN}(\sup_{w\in\supp[f^n=\Id]}[f^n=\Id](\{w\}))<\infty$,
so that
$([f^n=\Id]\times[f^n=\Id])(\diag_{\bP^1(\overline{k})})
 \le(d^n+1)\cdot\sup_{w\in\supp[f^n=\Id]}[f^n=\Id](\{w\})=O(d^n)$
 as $n\to\infty$.
 Now \eqref{eq:recurrence} follows from \eqref{eq:nonlin}. 
\end{proof}

\section{Proof of Theorem $\ref{th:discriminant}$}\label{sec:discriminant}

Let $k$ be a field and $k_s$ the separable closure of $k$ in $\overline{k}$.
Let $p(z)\in k[z]$ be a polynomial of degree $>0$ and $\{z_1,\ldots,z_m\}$
the set of all distinct zeros of $p(z)$ in $\overline{k}$ so that
$p(z)=a\cdot\prod_{j=1}^m(z-z_j)^{d_j}$ in $\overline{k}[z]$
for some $a\in k\setminus\{0\}$ and some sequence $(d_j)_{j=1}^m$ in $\bN$.
For a whole, we do not assume $\{z_1,\ldots,z_m\}\subset k_s$.
Let $\{p_1(z),p_2(z),\ldots,p_N(z)\}$ be the set of all 
 mutually distinct, non-constant, irreducible,
 and monic factors of $p(z)$ in $k[z]$ so that
$p(z)=a\cdot\prod_{\ell=1}^Np_\ell(z)^{s_\ell}$ in $k[z]$
 for some sequence $(s_\ell)_{\ell=1}^N$ in $\bN$.
 For every $\ell\in\{1,2,\ldots,N\}$, by the irreducibility of $p_\ell(z)$ in $k[z]$,
 $p_\ell(z)$ is the unique monic minimal polynomial in $k[z]$ of each zero of 
 $p_\ell(z)$ in $\overline{k}$,
 so $p_{\ell}(z)$ and $p_{n}(z)$ have no common zeros in $\overline{k}$
 if $\ell\neq n$. Hence
 for each $j\in\{1,2,\ldots,m\}$, 
 there is a unique $\ell=:\ell(j)\in\{1,2,\ldots,N\}$
 such that $p_\ell(z_j)=0$.
 Now suppose that $\{z_1,z_2,\ldots,z_m\}\subset k_s$.
 Then for every $\ell\in\{1,2,\ldots,N\}$, 
 $p_\ell(z)=\prod_{i:\ell(i)=\ell}(z-z_i)$ in $\overline{k}[z]$,
 so that
\begin{gather}
 d_i=s_{\ell(i)}\label{eq:degree} 
\end{gather}
for every $i\in\{1,2,\ldots,m\}$.
For every distinct $\ell,n\in\{1,2,\ldots,N\}$,
 \begin{gather}
 \textstyle\prod_{j:\ell(j)=\ell}\prod_{i:\ell(i)=n}(z_j-z_i)
 =\prod_{j:\ell(j)=\ell}p_n(z_j)=R(p_\ell,p_n),\label{eq:different}
 \end{gather}
 where $R(p,q)\in k$ is the (usual) resultant of $p(z),q(z)\in k[z]$. 
 The derivation $p'_\ell(z)$ of $p(z)$ in $k[z]$ satisfies
 $p_\ell'(z)=\sum_{h:\ell(h)=\ell}(\prod_{i:\ell(i)=\ell\text{ and }i\neq h}(z-z_i))$
 in $\overline{k}[z]$.
Hence for every $\ell\in\{1,2,\ldots,N\}$, 
 \begin{gather}
 \textstyle\prod_{j:\ell(j)=\ell}\prod_{i:\ell(i)=\ell\text{ and }i\neq j}(z_j-z_i)
 =\prod_{j:\ell(j)=\ell}p_{\ell}'(z_j)
 =R(p_{\ell},p_{\ell}').\label{eq:same}
 \end{gather}
 By \eqref{eq:degree}, \eqref{eq:same}, and \eqref{eq:different}, we have
 \begin{multline*}
 \textstyle D^*(p):=\prod_{j=1}^m\prod_{i:i\neq j}(z_j-z_i)^{d_id_j}
 =\prod_{j=1}^m\prod_{i:i\neq j}(z_j-z_i)^{s_{\ell(i)}s_{\ell(j)}}\\
  =\textstyle\prod_{\ell=1}^N(\prod_{j:\ell(j)=\ell}\bigl(\left(\prod_{i:\ell(i)=\ell\text{ and }i\neq j}(z_j-z_i)^{s_\ell^2}\right)\times\\
 \textstyle \times\left(\prod_{n:n\neq\ell}\prod_{i:\ell(i)=n}(z_j-z_i)^{s_ns_\ell}\right)\bigr))\\  
 =\textstyle\prod_{\ell=1}^N\left(R(p_\ell,p_\ell')^{s_\ell^2}\cdot
\prod_{n:n\neq\ell}R(p_\ell,p_n)^{s_ns_\ell}\right),
\end{multline*}
which is in $k\setminus\{0\}$. Now the proof is complete. \qed

\section*{Acknowledgement}

The author thanks the referee for a very careful scrutiny and
invaluable comments,
and also thanks
Professor Mamoru Asada for discussions on Theorem \ref{th:discriminant}. 
This work was partly done during the author's visiting National Taiwan Normal University and Academia Sinica, and the author thanks the institutes and
Professor Liang-Chung Hsia for their hospitality.
This work was partially supported by JSPS Grant-in-Aid for Young Scientists (B), 24740087.


\def\cprime{$'$}

\end{document}